\documentclass[article,dvisp]{amsart} 

\usepackage{graphicx}
\usepackage{amssymb}
\usepackage{epstopdf}
\usepackage{pspicture}
\usepackage{pstricks} 
\usepackage{pst-all}

\usepackage{tabularx}
\usepackage{bigstrut}
\usepackage{hyperref}

\usepackage{dsfont}
\usepackage{multirow}
\usepackage{amsfonts}
\usepackage{amssymb}
\usepackage{amsthm}
\usepackage{newlfont}
\usepackage{graphics} 
\usepackage{graphicx}
\usepackage{amscd}        
\usepackage[all]{xy}
\usepackage{pspicture}
\usepackage{pstricks}

\def\R{\mathbb{R}}

\def\P{\mathbb{P}}

\def\O{\Omega}

\newcommand{\n}{\mathbb{N}}
\newcommand{\ov}{\overline}

\newcommand\Tau{\mathcal{T}}

\newcommand{\zp}{\mathbb{Z}/p\mathbb{Z}\oplus\mathbb{Z}/p\mathbb{Z}}
\newcommand{\ds}{\displaystyle}
\def\Z{\mathbb{Z}}

\newtheorem{thm}{Theorem}[section]
\theoremstyle{definition}
\newtheorem{defn}[thm]{Definition} 
\newtheorem{ex}[thm]{Example}
\newtheorem{lem}[thm]{Lemma}
\newtheorem{cor}[thm]{Corollary}
\newtheorem{prop}[thm]{Proposition}
\newtheorem{rem}[thm]{Remark}
\theoremstyle{remark}

\numberwithin{equation}{section}

\theoremstyle{lemma}
\newtheorem{lema}{Lemma}[section]

\newtheorem{coro}[lema]{Corollary}

\newtheorem{teo}[lema]{Theorem}

\newtheorem{defi}[lema]{Definition}

\newtheorem{afin}[lema]{Claim}

\newtheorem{pro}[lema]{Proposition}

\newcommand{\Frac}{\displaystyle \frac}
\newcommand{\Sqrt}{\displaystyle \sqrt}

\begin{document}
\title[On NET Maps: Examples and Nonexistence Results]{On NET Maps: Examples and Nonexistence Results}
\author[E.A.~S\'aenz Maldonado]{Edgar A.~Saenz Maldonado}
\address{Department of Mathematics,
Virginia Tech,
Blacksburg, VA 24061, U.S.A.}
\email{easaenzm@math.vt.edu}
\begin{abstract} 
A Thurston map is called \begin{em}nearly Euclidean\end{em} if its local degree at each critical point is 2 and it has exactly four postcritical points. Nearly  Euclidean Thurston (NET) maps are simple generalizations of rational Latt\`{e}s maps. We investigate when such a map has the property that the associated pullback map on Teichm\"uller space is constant. We also show that no Thurston map of degree 2 has constant pullback map.\end{abstract}

\maketitle
\section{Introduction}
Let $S^2$ be a topological 2-sphere with a fixed orientation. We use $\P^{1}$ to denote the Riemann sphere. In this paper, all maps $S^2\to S^2$ will be orientation preserving. Let $f:S^2\to S^2$ be a branched cover and let $\O_{f}$ be the set of its critical points. We define the $postcritical$ $set$ $of$ $f$ to be
$$P_{f}:=\ds\bigcup_{n>0}f^{\circ n}(\O_{f}).$$

If $P_{f}$ is finite, we call $f$ a $Thurston$ $map$. Two Thurston maps $f$ and $g$ are called $equivalent$ iff there exist homeomorphisms $h_{0}:(S^2,P_{f})\to (S^2,P_{g})$ and $h_{1}:(S^2,P_{f})\to (S^2,P_{g})$ for which $h_{0}\circ f=g\circ h_{1}$ and $h_{0}$ is isotopic, rel $P_{f}$, to $h_{1}$.  In this case, if $g$ is a rational map we also say that $f$ is $realized$ by $g$.  

Suppose $f:S^2\to S^2$ is a Thurston map. The orbifold $\mathcal{O}_{f}=(S^{2},v_{f})$ associated to $f$ is the topological orbifold with underlying space $S^2$ and whose weight function $\nu_{f}(x)$ at $x$ is  given by $\nu_{f}(x)=\text{lcm}\{n\in\Z^{+}$: there exists a positive integer $m$ such that $f^{\circ m}$ has degree $n$ at some $y\in S^{2}$ with $f^{\circ m}(y)=x\}$. Let $\Tau_{f}$ be the Teichm\"uller space of $\mathcal{O}_{f}$. We may regard the space $\Tau_{f}$ as the space of complex structures on $\mathcal{O}_{f}$, up to the equivalence of isotopy fixing $P_{f}$. A complex structure on $\mathcal{O}_{f}$ pulls back under $f$ to a complex structure on $(S^2, f^{-1}(\nu_{f}))$, and this extends to a complex structure on $\mathcal{O}_{f}$. In this way we obtain a map $\Sigma_{f}: \Tau_{f}\to\Tau_{f}$. We will refer to $\Sigma_{f}$ as the $pullback$ $map$ induced by $f$.

In \cite{DH}, Douady and Hubbard, following Thurston, provide necessary and sufficient conditions for a Thurston map to be equivalent to a rational map.

\begin{thm}$\mathrm{(Thurston).}$ A Thurston map $f$ is equivalent to a rational map if and only if $\Sigma_{f}$ has a fixed point.\end{thm}

One would expect it to be rare (if it happens at all) for a Thurston map $f$ to have the pullback map $\Sigma_{f}$ be constant. In the Buff-Epstein-Koch-Pilgrim paper \cite{BEKP}, they give an example where the pullback map $\Sigma_{f}$ is constant and characterize when $\Sigma_{f}$ is constant. The example uses a result of McMullen (Proposition 5.1 on \cite{BEKP}) to construct a Thurston map $f$ with $\Sigma_{f}$ constant by having $\Sigma_{f}$ factor through a trivial Teichm\"uller space. That result is formulated as follows.

\begin{thm}$\mathrm{(McMullen).}\label{teo:12}$ Let $s:\P^{1}\to\P^{1}$ and $g:\P^{1}\to\P^{1}$ be rational maps with critical value sets $V_{s}$ and $V_{g}$. Let $A\subset\P^{1}$. Assume $V_{s}\subseteq A$ and $V_{g}\cup g(A)\subseteq s^{-1}(A)$. Then
\begin{itemize}
\item $f=g\circ s$ is a Thurston map,
\item $V_{g}\cup g(V_{s})\subseteq P_{f}\subseteq V_{g}\cup g(A)$ and
\item The dimension of the image of $\Sigma_{f}:\text{Teich}(\P^{1},P_{f})\to\text{Teich}(\P^{1},P_{f})$ is at most $|A|-3$.
\end{itemize}
\end{thm}

We refer to the assumptions of this Theorem as the \begin{em}McMullen's constant conditions\end{em}. We showed in \cite{S} that not every Thurston map whose Teichm\"uller map is constant satisfies the McMullen's constant conditions. The Teichm\" uller map associated to the rational map $f(z)=-\Sqrt[3]{2}z(z^3+2)/(2z^3+1)$ is constant (see Appendix D of \cite{S}). The ramification portrait for this map $f$ is: 
\[\xymatrix{
x \ar[r]^-{3} &-\Sqrt[3]{2}x \ \ar[rd]\\
y \ar[r]^-{3} &-\Sqrt[3]{2}y \ar[r] & 0\ar@(dr,ur)[]\\
z \ar[r]^-{3} &-\Sqrt[3]{2}z\ar[ru]\\ 
}\]

where $x=-1/2+\sqrt{3}i/2$, $y=-1/2-\sqrt{3}i/2$, and $z=1$. However, $f$ cannot be written as the composition of two maps because $\deg(f)=4$ and the local degree of $f$ at every critical point is $3$. Other examples of Thurston maps whose induced maps on Teichm\"uller space are constant and that do not satisfy the McMullen's constant conditions can be found in the class of NET maps (see Section 3). We are particularly interested in NET maps for which the induced pullback map is constant.
Theorem 10.2 of \cite{CFPP} provides an algebraic characterization of those NET maps whose induced maps on Teichm\"uller space are constant. This characterization reduces to the existence of $nonseparating\ sets$ for finite Abelian groups generated by two elements. Our main result is focused on this purely algebraic problem.
\begin{thm}\label{thm}$\mathrm{(Main\ Theorem).}$ Let $A$ be a finite Abelian group generated by two elements such that $A/2A\cong\mathbb{Z}/2\mathbb{Z}\oplus\mathbb{Z}/2\mathbb{Z}$. If 
$|A|=4p_{1}^{k_{1}}p_{2}^{k_{2}}\cdots p_{n}^{k_{n}}$ with  $p_{i}$ prime, $p_{i}\geq 13$ and $|k|=k_{1}+k_{2}+\cdots +k_{n}\geq1$, then $A$ does not contain a nonseparating subset.
\end{thm}

As a consequence of this, if $n$ is a positive integer with prime factorization $n=p_{1}^{k_{1}}p_{2}^{k_{2}}\cdots p_{n}^{k_{n}}$ where each $p_{i}$ is at least $13$, then there does not exist a NET map with degree $n$ whose Teichm\"uller map is constant.

This paper is organized as follows. Section 2 sets notation, reviews results and provides new algebraic properties related to $nonseparating\ sets$ which will be needed in the sequel. Section 3 introduces NET maps, reviews basic facts and applies the theory of Section 2 in the construction of two examples of NET maps whose Teichm\" uller maps are constant. One of these examples does not satisfy the McMullen's constant conditions. In Section 4 we investigate when the induced pullback map on Teichm\"uller space of NET maps cannot be constant. In Section 5 we show that no Thurston map of degree 2 has constant pullback map. 

The results described in this paper are part of the author's PhD thesis under the supervision of William Floyd. The author would like to thank William Floyd for his essential guidance and for his constant help, support and patience. The author also would like to thank Walter Parry for several helpful discussions, useful comments and for sharing his mathematical insights.

\section{Coset numbers and nonseparating sets}

In this section, we first review the definitions and facts on $coset\ numbers$ and $nonseparating\ sets$. Then we prove the converse of Lemma \ref{lem:56} and a technical lemma relevant in the proofs of the main results.

Let $A$ be a finite abelian group. Let $H$ be a subset of $A$ which is the disjoint union of four pairs $\{\pm h_{1}\},$ $\{\pm h_{2}\},$ $\{\pm h_{3}\},$ $\{\pm h_{4}\}$. (It is possible that $h_{i}=-h_{i}$.) Let $B$ be a subgroup such that $A/B$ is cyclic, and let $a$ be an element of $A$ so that $a+B$ generates $A/B$. Let $n$ be the order of $A/B$. For each $k\in\{1,2,3,4\}$ there exists a unique integer $c$ with $0\leq c\leq n/2$ such that $(ca+B)\cap\{\pm h_{k}\}\neq\emptyset$.
Let $c_{1},c_{2},c_{3},c_{4}$ be these four integers ordered so that $0\leq c_{1}\leq c_{2}\leq c_{3}\leq c_{4}$. These four numbers are called \begin{em}coset numbers\end{em} for $H$ relative to $B$ and the generator $a+B$ of $A/B$.

Let $A$ be a finite Abelian group. A subset $H$ of $A$ is called \begin{em}nonseparating\end{em} if and only if it satisfies the following conditions:
\begin{itemize}
\item $H$ is a disjoint union of the form $H=H_{1}\coprod H_{2}\coprod H_{3}\coprod H_{4}$, where each $H_{i}$ has the form $H_{i}=\{\pm h_{i}\}$. (It is possible that $h_{i}=-h_{i}$.)                          
\item Let $B$ be a cyclic subgroup of $A$ such that $A/B$ is cyclic. Let  $c_{1},c_{2},c_{3},c_{4}$ be the coset numbers for $H$ relative to $B$ and some generator of $A/B$. The main condition is that $c_{2}=c_{3}$ for every such choice of $B$ and generator of $A/B$.
\end{itemize}

\begin{ex}\label{ex:1}  Let $A=\Z/3\Z\oplus\Z/3\Z$. The subset $H=A\setminus\{(0,0)\}$ is a nonseparating subset of $A$. In fact, let $B$ be a cyclic subgroup of $A$ so that $A/B$ is cyclic. Then $B\cong A/B\cong\Z/3\Z$. Given a generator $a+B$ of $A/B$ we have only three cosets: $B$, $a+B$ and $2a+B$. It is obvious that $B$ contains exactly one pair of mutually inverse elements of order 3. So $c_{1}=0$ and $c_{2}=c_{3}=c_{4}=1$.\end{ex} 

\begin{ex}\label{ex:3}  In this example we show that $H_{1}=\{\pm(1,0),\pm(0,1),\pm(1,2),\pm(2,1)\}$ is a nonseparating subset of $A=\Z/4\Z\oplus\Z/4\Z$. Let $B$ be a cyclic subgroup of $A$ such that $A/B$ is cyclic. Then $B\cong A/B\cong \Z/4\Z$. There are only six possible choices for $B$: $\left<(1,0)\right>$,  $\left<(0,1)\right>$, $\left<(1,2)\right>$, $\left<(2,1)\right>$, $\left<(3,1)\right>$, and  $\left<(1,1)\right>$. If $B=\left<(1,1)\right>$ or $B=\left<(1,3)\right>$, one verifies in these cases that $c_{1}=c_{2}=c_{3}=c_{4}=1$. If $B\neq\left<(1,1)\right>$ and $B\neq\left<(1,3)\right>$, one verifies in these cases that $c_{1}=0$ and $c_{2}=c_{3}=1$.
\end{ex} 

\begin{ex}\label{ex:4}  In this example we show that $H_{2}=\{\pm(1,0),\pm(0,1),\pm(1,1),\pm(1,3)\}$ is also a nonseparating subset of $A=\Z/4\Z\oplus\Z/4\Z$. If $B=\left<(1,2)\right>$ or $B=\left<(2,1)\right>$, one verifies in these cases that $c_{1}=c_{2}=c_{3}=1$. If $B\neq\left<(1,2)\right>$ and $B\neq\left<(2,1)\right>$, one verifies in these cases that $c_{1}=0$ and $c_{2}=c_{3}=1$.
\end{ex} 

The next two lemmas provide ways to produce nonseparating subsets from known ones. For details of the proofs, see Section 10 of \cite{CFPP}.

\begin{lem}\label{lem:55} Let $A$ be a finite Abelian group, and let $H$ be a nonseparating subset of $A$. If $\varphi:A\to A$ is a group automorphism and $h$ is an element of order 2 in $A$,  then $\varphi(H)+h$ is a nonseparating subset of $A$.\end{lem} 

\begin{lem}\label{lem:56} If $A$ is a finite Abelian group and if $A'$ is a subgroup of $A$, then every subset of $A'$ which is nonseparating for $A'$ is nonseparating for $A$. 
\end{lem}

\begin{ex}\label{ex:2} Let $A=\Z/4\Z\oplus\Z/2\Z$. The set $H=\{(0,0),\pm(1,0),\pm(2,0),\pm(1,1)\}$ is a nonseparating subset of $A$. For details of the proof see Example 10.3 of  \cite{CFPP}. By Lemma \ref{lem:56}, $H_{3}=\{(0,0),\pm(1,0),\pm(2,0),\pm(1,2)\}$ is a nonseparating subset of $\Z/4\Z\oplus\Z/4\Z$.\end{ex} 

In an unpublished result, Walter Parry found that up to automorphisms followed by a translation by an element of order 2, $H_{1}$, $H_{2}$ and $H_{3}$ are the only distinct nonseparating subsets of $\Z/4\Z\oplus\Z/4\Z$.

\begin{ex}\label{ex:5} Let $k$ be an integer with $k\geq3$. Let $A=\Z/2^{k}\Z\oplus\Z/2\Z$ and let $H=\{\pm(1,0),\pm(2^{k-2},0),\pm(2^{k-2},1),\pm(2^{k-1}-1,0)\}$. We show that $H$ is a nonseparating subset of $A$. Let $B$ be a cyclic subgroup of $A$ such that $A/B$ is cyclic. Then either $|B|=2^{k}$ or $|B|=2$. First suppose that $|B|=2^{k}$. Then $A/B\cong\Z/2\Z$. Given a generator $a+B$ of $A/B$ we have only two cosets: $B$ and $a+B$. In this case to show that $H$ does not separate $c_{2}$ from $c_{3}$ it suffices to prove that $B$ does not contain exactly two elements of $H$. If $(1,0)\in B$, then $(2^{k-2}, 0)$ and $(2^{k-2}-1,0)\in B$. The same is true if $(2^{k-2}-1,0)\in B$. So if $B$ contains exactly two elements of $H$, then these elements are $(2^{k-2},0)$ and $(2^{k-2},1)$. But then $(0,1)\in B$. This is impossible.

Now suppose that $|B|=2$. Then either $B=\left<(0,1)\right>$ or $B=\left<(2^{k-1},1)\right>$. Let $a\in A$ such that $a+B$ generates $A/B$. The first component of $a$ has the form $4r\pm1$ for some integer $r$. Hence $2^{k-2}a+B=\pm(2^{k-2},0)+B$, and so the coset number of $\pm(2^{k-2},0)$ is $2^{k-2}$. Similarly, one verifies that the coset number of $\pm(2^{k-2},1)$ is $2^{k-2}$. Let $m$ be the integer in $\{0,\cdots,2^{k-1}\}$ such that $m(a+B)=\pm(1,0)+B$. Then $(2^{k-1}-m)(a+B)=\pm(2^{k-1}-1,0)+B$. So if $m$ is the coset number of $\pm(1,0)$, then $2^{k-1}-m$ is the coset number of $\pm(2^{k-1}-1,0)$. Thus, $\{c_{1},c_{4}\}=\{m, 2^{k-1}-m\}$ and $c_{2}=c_{3}=2^{k-2}$. This proves that $H$ is a nonseparating subset of $A$.\end{ex}

\begin{rem} Let $A$, $H$ be as in Example \ref{ex:5}. The group $A$ contains the subgroup $\left<(2^{k-2},0)\right>\oplus\Z/2\Z$ which is isomorphic to $\Z/4\Z\oplus\Z/2\Z$. By Lemma \ref{lem:56} and Example \ref{ex:2} above, $H'=\{(0,0),\pm(2^{k-2},0),\pm(2^{k-1},0),\pm(2^{k-2},1)\}$ is also a nonseparating subset of $A$. If there were an element $h$ of order 2 in $A$ and an automorphism $\varphi$ of $A$ such that $H=\varphi(H')+h$, then $h=\varphi(0)+h$ would be an element of $H$. However, $H$ contains no element of order $2$. So $H$ cannot be gotten from $H'$ by an automorphism of $A$ together with a translation by an element of order $2$ in $A$.\end{rem}

The next lemma shows the converse of Lemma \ref{lem:56}. For additional details of the proof, see Appendix A.
\begin{lem}\label{lem:a05} Let $A$ be a finite Abelian group generated by two elements and let $A'$ be a subgroup of $A$. If $H$ is a subset of $A'$ which is nonseparating for $A$, then $H$ is nonseparating for $A'$.
\end{lem}
\begin{proof}
Let $B'$ be a cyclic subgroup of $A'$ such that $A'/B'$ is cyclic. Let $a'$ be an element of $A'$ such that $a'+B'$ generates $A'/B'$. By Proposition \ref{prop:a01}, there exists $B$ a subgroup of $A$ such that  $A/B$ is cyclic and $A'\cap B=B'$. Let $n$ be the order of $A/B$. Let $m$ be the order of $a'+B\in A/B$. By Proposition \ref{prop:a1}, there exists an element $a$ in $A$ such that $a+B$ generates $A/B$ and  $(n/m)(a+B)=a'+B$. Let $0\leq c_{1}\leq c_{2}\leq c_{3}\leq c_{4}\leq(1/2)|A'/B'|$ be the coset numbers for $H$ relative to $B'$ and the generator $a'+B'$ of $A'/B'$. Since $a'\in A'$ and $ma'\in B$, $ma'\in B'$. So $|A'/B'|$ divides $m$. This yields, $0\leq nc_{1}/m\leq nc_{2}/m\leq nc_{3}/m\leq nc_{4}/m\leq n/2$. Since $c_{i}(a'+B')\subseteq c_{i}(a'+B)=(nc_{i}/m)(a+B)$, it follows that $nc_{1}/m, nc_{2}/m, nc_{3}/m, nc_{4}/m$ are the coset numbers for $H$ relative to $B$ and $a+B$. Hence $nc_{2}/m=nc_{3}/m$, and so $c_{2}=c_{3}$.\end{proof}

\begin{lem}\label{lem:a6} Let $a$ and $b$ be odd positive integers such that $a|b$ and $a>1$. Let $A=\Z/2\Z\oplus\Z/2\Z\oplus\Z/a\Z\oplus\Z/b\Z$ and let $\phi:A\to\Z/a\Z\oplus\Z/b\Z$ be the canonical projection.
Suppose that $A$ contains a nonseparating subset $H=\coprod_{i=1}^{4}\{\pm h_{i}\}$. Let $D=\{\phi(h_{i})\pm\phi(h_{j}): i,j\in\{1,2,3,4\}\ \text{with}\ \ i<j\}$. Assume that there exists a cyclic subgroup $G$ of $\Z/a\Z\oplus\Z/b\Z$ such that $G\cap D\subseteq\{0\}$ and $(\Z/a\Z\oplus\Z/b\Z)/G$ is cyclic. If $\Z/a\Z\oplus\Z/b\Z=\langle\phi(H)\rangle$, then we may assume that $\langle\phi(h_{1}),\phi(h_{2})\rangle=\Z/a\Z\oplus\Z/b\Z$ and  $h_{2}$, $h_{3}$ and $h_{4}$ all differ by an element of order 2.\end{lem}
\begin{proof} Define the following three cyclic subgroups of $A$:
 \begin{itemize}
 \item $E_{(1,0)}=\langle (1,0)\rangle\oplus G$
 \item $E_{(0,1)}=\langle (0,1)\rangle\oplus G$
 \item $E_{(1,1)}=\langle (1,1)\rangle\oplus G$
 \end{itemize}
 
The groups $A/E_{(1,0)}$, $A/E_{(0,1)}$ and $A/E_{(1,1)}$ are all cyclic. For each subgroup $B$ of $A$ such that $A/B$ is cyclic, we denote the second coset number for $H$ relative to $B$ and the generator $a+B$ by $c_{2}$. 

Let $w+E_{(1,0)}$ be a generator of $A/E_{(1,0)}$. Since $H$ is a nonseparating subset of $A$, without loss of generality we may assume that $h_{2}$ and $h_{3}$ are elements of $c_{2}w+E_{(1,0)}$. Then $\phi(h_{2})-\phi(h_{3})\in G$. So, $\phi(h_{2})=\phi(h_{3})$ and $h_{2}-h_{3}=(1,0,0,0)$. 

Let $x+E_{(0,1)}$ be a generator of $A/E_{(0,1)}$. We show that $\{\pm h_{2},\pm h_{3}\}\subset c_{2}x+E_{(0,1)}$ cannot occur. Proceed by contradiction.  If $\{h_{2},h_{3}\}\subset c_{2}x+E_{(0,1)}$ or $\{-h_{2},-h_{3}\}\subset c_{2}x+E_{(0,1)}$, then $h_{2}-h_{3}=(0,1,0,0)$ which is impossible. If $\{h_{2},-h_{3}\}\subset c_{2}x+E_{(0,1)}$ or $\{-h_{2},h_{3}\}\subset c_{2}x+E_{(0,1)}$, then $h_{2}+h_{3}=(0,1,0,0)$ and so $2h_{2}=(1,1,0,0)$, which yields a contradiction. Similarly, if $y+E_{(1,1)}$ is a generator of $A/E_{(1,1)}$, then $\{\pm h_{2},\pm h_{3}\}\subset c_{2}y+E_{(1,1)}$ cannot occur.
 
We now show that $\{\pm h_{1},\pm h_{4}\}\subset c_{2}x+E_{(0,1)}$ cannot occur. Relabeling, if necessary, it suffices to show that $\{h_{1}, h_{4}\}\subset c_{2}x+E_{(0,1)}$ cannot occur. Proceed by contradiction. Suppose that $\{h_{1},h_{4}\}\subset c_{2}x+E_{(0,1)}$  then $\phi(h_{1})=\phi(h_{4})$ and $h_{1}-h_{4}=(0,1,0,0)$. Let $y+E_{(1,1)}$ be a generator of  $A/E_{(1,1)}$. Then one of the following sixteen inclusions must hold.
\begin{itemize}
\item $\{\pm h_{1},\pm h_{2}\}\subset c_{2}y+E_{(1,1)}$
\item $\{\pm h_{1},\pm h_{3}\}\subset c_{2}y+E_{(1,1)}$
\item $\{\pm h_{4},\pm h_{2}\}\subset c_{2}y+E_{(1,1)}$
\item $\{\pm h_{4},\pm h_{3}\}\subset c_{2}y+E_{(1,1)}$
\end{itemize}
However, each of them would imply that $\langle\phi(H)\rangle=\langle\phi(h_{1})\rangle$. Since $\Z/a\Z\oplus\Z/b\Z=\langle\phi(H)\rangle$, none of these inclusions occur. So, $\{\pm h_{1},\pm h_{4}\}\subset c_{2}x+E_{(0,1)}$ cannot occur. Similarly, if $z+E_{(1,1)}$ is a generator of $A/E_{(1,1)}$, then $\{\pm h_{1},\pm h_{4}\}\subset c_{2}z+E_{(1,1)}$ cannot occur either.
 
Now, we may assume that either $\{h_{3},h_{4}\}$ or $\{-h_{3},h_{4}\}$ is a subset of $c_{2}x+E_{(0,1)}$. If $\{h_{3},h_{4}\}\subset c_{2}x+E_{(0,1)}$, then $h_{3}-h_{4}=(0,1,0,0)$. Hence $\phi(h_{2})=\phi(h_{3})=\phi(h_{4})$ and the lemma follows. If $\{-h_{3},h_{4}\}\subset c_{2}x+E_{(0,1)}$, then $h_{3}+h_{4}=(0,1,0,0)$. In this case, $\phi(h_{2})=\phi(h_{3})=-\phi(h_{4})$. Relabeling $h_{4}$ by $-h_{4}$, the lemma follows.
\end{proof}

\section{NET Maps: Preliminaries and Examples}
In this section we review briefly some definitions and properties of NET maps. We refer the reader to Section 1 in \cite{CFPP} for more details.
\begin{defn} A Thurston map $f:S^2\to S^2$ is called $Euclidean$ if its local degree at each of its critical points is $2$, it has at most four postcritical points and none of them is critical.
\end{defn}
\begin{defn} A Thurston map $f:S^2\to S^2$ is called $nearly$ $Euclidean$ (NET) if its local degree at each of its critical points is $2$ and it has exactly four postcritical points. 
\end{defn}

From Lemma 1.3 of \cite{CFPP}, it follows that every Euclidean Thurston map is nearly Euclidean, and every NET map $f$ has the property that $f^{-1}(P_{f})$ contains exactly four points which are not critical points; $f$ is Euclidean if and only if these four points are precisely the points of $P_{f}$. The next theorem shows that NET maps lift to maps of tori. The proof of the theorem and the following description can be found in Section 1 of \cite{CFPP}.

\begin{thm}\label{teo:31} Let $f$ be a Thurston map. Then $f$ is nearly Euclidean if and only if there exist branched covering maps $p_{1}:T_{1}\to S^2$ and $p_{2}:T_{2}\to S^2$ with degree 2 from the tori $T_{1}$ and $T_{2}$ to $S^2$ such that the set of branch of $p_{2} $ is $P_{f}$ and there exists a continuous map $\tilde{f}: T_{1}\to T_{2}$ such that $p_{2}\circ\tilde{f}=f\circ p_{1}$. If $f$ is nearly Euclidean, then $f$ is Euclidean if and only if the set of branched points of $p_{1}$ is $P_{f}$. 
\end{thm}

Let $f:S^2\to S^2$ be a NET map. Let $p_{1}:T_{1}\to S^2$, $p_{2}:T_{2}\to S^2$ and $\tilde{f}: T_{1}\to T_{2}$ as in Theorem \ref{teo:31} such that $p_{2}\circ\tilde{f}=f\circ p_{1}$. For $j\in\{1,2\}$, let $P_{j}(f)\subset S^2$ be the set of branched points of $p_{j}$  and let $q_{j}:\R^2\to T_{j}$ be a universal covering map. The map $p_{j}\circ q_{j}:\R^2\to S^2$ is a regular branched covering map whose local degree at every ramification point is $2$. Let $\Gamma_{j}$ and $\Lambda_{j}$ be the set of deck transformations and the set of ramification points of $p_{j}\circ q_{j}$. We can choose $q_{j}$ so that $\Gamma_{j}$ is generated by the set of all Euclidean rotations of order $2$ about the points of $\Lambda_{j}$. We may, and do, normalize so that $0\in\Lambda_{j}$. Hence $\Lambda_{j}$ is a lattice in $\R^2$ and the elements of $\Gamma_{j}$ are the maps of the form $x\mapsto2\lambda\pm x$ for some $\lambda\in\Lambda_{j}$. 

The map $\tilde{f}\circ q_{1}$ lifts to a continuous map $\hat{f}:\R^2\to\R^2$ such that $q_{2}\circ\hat{f}=\tilde{f}\circ q_{1}$. Since $\tilde{f}\circ q_{1}$ is a covering map, the map $\hat{f}$ also is. Hence, $\hat{f}$ is a homeomorphism. We replace $q_{1}$  by $q_{1}\circ\hat{f}^{-1}$. In this case, $\tilde{f}$ lifts to the identity map. Thus, $\Lambda_{1}\subseteq\Lambda_{2}$ and $\Gamma_{1}\subseteq\Gamma_{2}$. So we obtain the standard commutative diagram
$$\begin{CD}
\Lambda_{1}@>{i_{c}}>>\Lambda_{2}\\
{i_{c}}@VVV      @VVV{i_{c}}\\
\R^2@>id>>\R^2\\
{q_{1}}@VVV      @VVV{q_{2}}\\
T_{1}@>{\tilde{f}}>>T_{2}\\
{p_{1}}@VVV      @VVV{p_{2}}\\
S^2@>f>>S^2
\end{CD}$$
where $id:\R^2\to\R^2$ is the identity map and the maps from $\Lambda_{1}$ to $\Lambda_{2}$ are inclusion maps.

The group $\Gamma_{j}$ contains the group of deck transformations of $q_{j}$. It is the subgroup with index 2 consisting of translations of the form $x\mapsto 2\lambda+x$ with $\lambda\in\Lambda_{j}$. So we can identify $T_{j}$ with $\R^2/2\Lambda_{j}$. The standard commutative diagram implies that $\R^2/\Gamma_{1}$ and $\R^2/\Gamma_{2}$ are both $S^2$ identified with $S^{2}$. Thus there is an identification map $\phi:\R^2/\Gamma_{2}\to\R^2/\Gamma_{1}$. To evaluate $f$ at some point $x$, we view $x$ as an element of $\R^2/\Gamma_{1}$. We lift it to $\R^2/\Gamma_{2}$ and then apply the identification map $\phi$ to obtain $f(x)$. For Euclidean NET maps the identification map $\phi$ can be obtained by using $\Phi:\R^2\to\R^2$ an affine automorphism such that $\Phi(\Lambda_{2})=\Lambda_{1}$.

The following theorem shows that a NET map $f$ can be obtained by taking a Euclidean Thurston map $g$ and post-composing it by a homeomorphism $h$ that satisfies $h(P_{g})\subseteq g^{-1}(P_{g})$, subject to the constraint that the composition $f=h\circ g$ has four postcritical points. More precisely,

\begin{thm}\label{teo:34} (1) If $g:S^{2}\to S^{2}$ is a NET map and $h:S^{2}\to S^{2}$ is an orientation-preserving homeomorphism such that $h(P_{g})\subseteq g^{-1}(P_{g})$, then $f=h\circ g$ is a NET map if it has at least four post critical points.\\
(2) Let $f$ be a NET map with $P_{1}=P_{1}(f)$ and $P_{2}=P_{2}(f)$. Let $h:S^{2}\to S^{2}$ orientation-preserving homeomorphism with $h(P_{1})=P_{2}$. Then $f=h\circ g$, where $g:S^2\to S^2$ be a Euclidean Thurston map with $P_{g}=P_{1}$ and $P_{2}\subseteq g^{-1}(P_{g})$, so that $h(P_{g})\subseteq g^{-1}(P_{g})$. 
\end{thm}

Combining Theorem \ref{teo:31}, the description of the standard commutative diagram given above and Theorem \ref{teo:34}, each NET map can be constructed as follows. Let $\Lambda_{2}$ be the lattice generated by $(1,0)$ and $(0,1)$. Fix a quintuple $(\Lambda_{1}, \Phi, R, h_{R}, h)$, where 
\begin{itemize}
\item $\Lambda_{1}$ is a sublattice of $\Lambda_{2}$ of covolume greater than one. Let $\Gamma_{1}, \Gamma_{2}$ be the groups generated by rotations of order 2 about elements of $\Lambda_{1}$, $\Lambda_{2}$ so that $S^{2}_{1}:=\R^{2}/\Gamma_{1}$, $S^{2}_{2}:=\R^{2}/\Gamma_{2}$ are spheres. Since $\Lambda_{1}<\Lambda_{2}$, we have $\Gamma_{1}<\Gamma_{2}$. So the identity map 
$i:\R^{2}\to\R^{2}$ induces a quotient map $\overline{i}:S^{2}_{1}\to S^{2}_{2}$. This quotient map is a branched covering map of degree $d:=[\Lambda_{2}:\Lambda_{1}]$.
\item For $j\in\{1,2\}$, set $T_{j}:=\R^{2}/2\Lambda_{i}$ and let $p_{j}:T_{j}\to S^{2}_{j}$ be branched covering maps of degree 2. Let $P_{j}=\Lambda_{j}/\Gamma_{j}$ be the set of branched points of $p_{j}$. Note that $|P_{j}|=4$.
\item $\Phi:(\R^2,\Lambda_{2})\to(\R^2,\Lambda_{1})$ is an affine map of the form $\Phi(x)=Lx+b$ where $b\in\Lambda_{1}$ and $L$ is a $2\times2$ matrix over $\Z$ of determinant greater than one.
\item $R\subset\Lambda_{2}/\Gamma_{1}\subset S^{2}_{1}$ is a set of four points.
\item $h_{R}:P_{1}\to R$ is a bijection.
\item $h:(S^{2}_{1},P_{1})\to(S^{2}_{1}, R)$ is an orientation-preserving homeomorphism which is an extension of $h_{R}$.
\end{itemize}
$$
\xymatrix{
\R^{2} \ar[d]_{/\Gamma_{1}} \ar[r]^{i} &\R^{2}\ar[d]_{/\Gamma_{2}}  \ar[r]^{\Phi} &\R^{2} \ar[d] \ar[d]_{/\Gamma_{1}}\\
S^{2}_{1} \ar[r]^{\overline{i}} & S^{2}_{2}\ar[r]^{\phi} &S^{2}_{1} \ar[r]^{h} &S^{2}_{1} }
$$
The affine map $\Phi$ descends to a homeomorphism $\phi:S^{2}_{2}\to S^{2}_{1}$, so that the composition $g=\phi\circ\overline{i}$
is a Euclidean Thurston map such that $P_{g}=P_{1}$. Since $h(P_{g})=h(P_{1})=R$, it turns out that $f:=h\circ g$ is a Thurston map such that $P_{f}\subset R$. If $V_{g}=P_{1}$ (which is always true if $\deg(g)=3$ or $\deg(g)\geq5$), then $P_{f}=R$ and we get an NET map by this process. Thus, given a nearly Euclidean Thurston map $f$, we can always associate to $f$ a quintuple $(\Lambda_{1}, \Phi, R, h_{R}, h)$ and a commutative diagram as above where $f=h\circ g=h\circ(\phi\circ \overline{i})$. 

Under these settings, in \cite{CFPP} J. Cannon et al. proved the following result.

\begin{thm}\label{thm:55} Let $f$ be a NET and let $p_{1}, \Lambda_{1}, \Lambda_{2}$ as above. Then the Teichm\"uller map of $f$ is constant if and only if $p_{1}^{-1}(P_{f})$ is a nonseparating subset of $\Lambda_{2}/2\Lambda_{1}$.
\end{thm}

So, in order to construct NET maps whose Teichm\"uller maps are constant, we may consider the following steps:
\begin{itemize}
\item[1.] Let $\Lambda_{2}$ be the lattice generated by $(1,0)$ and $(0,1)$.

\item[2.] Construct a finite Abelian group $A$ generated by two elements with $A/2A\cong \Z/2\Z\oplus\Z/2\Z$ such that $A$ has a nonseparating subset $H$.

\item[3.] Construct a lattice $\Lambda_{1}$ such that $\Lambda_{1}<\Lambda_{2}$ for which $\Lambda_{2}/2\Lambda_{1}\cong A$.

\item[4.] Construct an isomorphism $\Phi$ from $\Lambda_{2}$ to $\Lambda_{1}$, which in effect produces a Euclidean Thurston map $g$ corresponding to $\Lambda_{1}$ and $\Lambda_{2}$. That is, $g=\phi\circ\overline{i}$.

\item[5.] Construct an orientation-preserving homeomorphism $h:S^{2}_{1}\to S^2_{1}$ such that $h(P_{g})=p_{1}(H)$. Here $p_{1}:T_{1}\to S^{2}_{1}$.

\item[6.] Set $f:=h\circ g$. By Theorem \ref{teo:34}, if $f$ has four postcritical points then it is a NET map. In that case, $P_{f}=h(P_{g})$ and so $H=p_{1}^{-1}(P_{f})$. Then, by Theorem \ref{thm:55},  the Teichm\"uller map of $f$ is constant. Since $|\Lambda_{2}/\Lambda_{1}|=\text{deg}(f)$, we have $|A|=4\deg(f)$.
\end{itemize}

In \cite{CFPP},  J. Cannon et al. prove a general existence theorem. If $d$ is an integer with $d>2$ such that $d$ is divisible by either $2$ or $9$, then there exists a NET map with degree $d$ whose  Teichm\"uller pullback map is constant. In particular, it is possible to construct NET maps with odd degree and constant  Teichm\"uller map.
 
\begin{lem}\label{lem:deg} Let $s:S^2\to S^2$ be an orientation-preserving branched covering map such that $\deg(s,x)=2$ for every $x\in\Omega_{s}$. If $|V_{s}|\leq 3$ then $\deg(s)=2$ or $\deg(s)=4$. 

\end{lem}
\begin{proof}
If $\deg(s)=3$, the preimage under $s$ of every element of $V_s$ contains three points counting multiplicity and no such preimage contains two critical points. Then $s$ maps its four critical points bijectively to $V_s$ and so $|V_s|=4$.
 
Now suppose that $\deg(s)\geq 5$. If $\deg(s)=2k$ for some integer $k$ with $k\geq3$, then by the Riemann-Hurwitz formula $s$ has exactly $4k-2$ distinct critical points; and the preimage of each critical value contains at most $k$ critical points. If $|V_{s}|\leq3$, then there are at most $3k$ distinct critical points. This would imply that $4k-2\leq 3k$ which leads to a contradiction. If $\deg(s)=2k+1$ for some integer $k$ with $k\geq2$, then by the Riemann-Hurwitz formula $s$ has exactly $4k$ distinct critical points; and the preimage of each critical value contains at most $k$ critical points. If $|V_{s}|\leq3$, as in the even case, we reach a contradiction.
 \end{proof}

\begin{prop}\label{prop:37} There exist NET maps with constant pullback map that do not satisfy the McMullen's constant conditions.
\end{prop}
\begin{proof} Let $f$ be an NET map with odd degree whose pullback map is constant. If $f$ satisfies the McMullen's constant conditions (see Theorem \ref{teo:12}), then there are two orientation-preserving branched covering maps $g$ and $s$ and a set $A$ such that $f=g\circ s$, $|A|\leq3$ and $V_{s}\subseteq A$. Since $f$ is nearly Euclidean, $\deg(s;x)=2$ for every $x\in\Omega_{s}$. By Lemma $\ref{lem:deg}$, either $\deg(s)=2$ or $\deg(s)=4$. This is impossible because $\deg(f)$ is an odd number.\end{proof}


\begin{ex} We construct an expanding rational NET map with degree 4 whose Teichm\"uller map is constant. Let $g$ be the Latt\'es map described in Example 9.17 of \cite{CFPP}.  The critical points of $g$ are $E_{1},E'_{1},E_{2},E'_{2},E_{3},E'_{3}$ and the postcritical set of $g$ is $\{e_{1},e_{2},e_{3},\infty\}$. Moreover, $g(e_{i})=g(\infty)=\infty$ and $g(E_{i})=g(E_{i}^{'})=e_{i}$ for $i\in\{1,2,3\}$. Now let $h:\P^{1}\to\P^{1}$ be an orientation-preserving homeomorphism such that $h(e_{1})=e_{1},\ h(e_{2})=E_{1},\ h(e_{3})=E_{1}^{'}, \ h(\infty)=\infty.$ 
Following the description given in Example 9.17 of \cite{CFPP}, one sees that $f=h\circ g$ is a NET map whose Teichm\"uller map is constant so it is combinatorially equivalent to a rational map, $R$. One easily verifies that $R(E_{1})=R(E_{1}^{'})=e_{1}$ mapping with degree $2$, $R(E_{2})=R(E_{2}^{'})=E_{1}$ mapping with degree $2$,
$R(E_{3})=R(E_{3}^{'})=E_{1}^{'}$ mapping with degree $2$
$R(e_{1})=R(\infty)=\infty$ mapping with degree $1$. So $R$ is a rational map without periodic critical points. Let $\mu$ be the Mobi\"us transformation that satisfies $\mu(E_{1})=0$, $\mu(E_{1}^{'})=\infty$, $\mu(e_{1})=1$ and let $F=\mu\circ R\circ\mu^{-1}$. Set $\alpha=\mu(E_{2}), \beta=\mu(E_{2}^{'}), \gamma=\mu(E_{3}), \lambda=\mu(E_{3}^{'})$. Using the branching data of $F$, we have that
$$F(z)=\Frac{(z-\alpha)^{2}(z-\beta)^2}{(z-\gamma)^2(z-\lambda)^2}.$$
The numerator of $F(z)-1$ is given by $(z-\alpha)^{2}(z-\beta)^2-(z-\gamma)^2(z-\lambda)^2$. Since $z=0$ is the unique zero of the numerator of $F(z)-1$ and has multiplicity 2, it follows that $(z-\alpha)(z-\beta)+(z-\gamma)(z-\lambda)= 2z^2$
and that $(z-\alpha)(z-\beta)-(z-\gamma)(z-\lambda)$ is a nonzero constant. Then, 
$$-(\alpha+\beta)+(\gamma+\lambda)=0,$$
$$-(\alpha+\beta)-(\gamma+\lambda)=0,$$
$$\alpha\beta+\gamma\lambda=0.$$
This forces $\alpha+\beta=0$ and $\gamma+\lambda=0$. So $\alpha=-\beta$, $\gamma=-\lambda$ and $-\beta^{2}-\lambda^{2}=0$. Hence, $F(z)$ has the form
$$F(z)=\Big(\Frac{z^{2}-\beta^{2}}{z^{2}+\beta^{2}}\Big)^{2}$$\\ with the restriction $F(F(1))=F(1)$. This restriction implies that either
$$\Frac{(F(1))^{2}-\beta^{2}}{(F(1))^{2}+\beta^{2}}=\Frac{1-\beta^{2}}{1+\beta^{2}} \ \ \ \text{or}\ \ \ \ \Frac{(F(1))^{2}-\beta^{2}}{(F(1))^{2}+\beta^{2}}=-\Frac{1-\beta^{2}}{1+\beta^{2}}.$$
Thus, either $(F(1))^2=1$ or $(F(1))^2=\beta^4$. If $(F(1))^2=1$, then $F(1)$ has to be $-1$. Hence, $\Big(\Frac{1-\beta^{2}}{1+\beta^{2}}\Big)^{2}=-1$
and so $\beta^2=\pm i$. If $(F(1))^2=\beta^4$, then $\Big(\Frac{1-\beta^{2}}{1+\beta^{2}}\Big)^{4}=\beta^4$. 

We now show that if $\beta^{2}=i$, then the map $F(z)=(z^{2}-i)^2/(z^{2}+i)^{2}$ is an expanding rational (see \cite{BM}) NET map with constant pullback map. To see this, it suffices to consider $g(w)=(w-i)^{2}/(w+i)^{2}$, $s(z)=z^2$, and $A=\{0,1,\infty\}$. Note that $F=g\circ s$ and $g,s$ and $A$ verify the McMullen's constant conditions.\end{ex} 

\begin{rem} The preceding example suggests the following family of rational maps: $F_{n}(z)=(z^n-i)^{2}/(z^n+i)^{2}$, with $n\geq2$. Each $F_{n}$ has no periodic critical points, so each $F_{n}$ is expanding. Also, note that $F_{n}=g\circ s_{n}$ where $g(w)=(w-i)^2/(w+i)^{2}$ and $s_{n}(z)=z^n$. If $n$ is even and $A=\{0,1,\infty\}$, then $g,s_{n}$ and $A$ verify the McMullen's constant conditions. Thus, if $n$ is even, $F_{n}$ has constant Teichm\"uller map.\end{rem}

\begin{ex}\label{ex:} We construct a NET map of degree 9 whose Teichm\"uller map is constant. First of all, we construct an Abelian group of degree $4\cdot 9=36$ which contains a nonseparating subset. We take $A=\Z/6\Z\oplus\Z/6\Z$. The 3-torsion subgroup of $A$ is isomorphic to $\Z/3\Z\oplus\Z/3\Z$. Example \ref{ex:1} and Lemma \ref{lem:56} imply that $H=\{\pm(0,2),\pm(2,2),\pm(2,4),\pm(2,0)\}$ is a nonseparating subset of $A$.
Let $\Lambda_{2}=\Z^2$ and let $\Lambda_{1}=3\Lambda_{2}$. So $\Lambda_{2}/2\Lambda_{1}\cong A$. Let $\Gamma_{j}$ be the group generated by rotations of order 2 about elements of $\Lambda_{j}$. Note that $S^{2}_{j}:=\R^{2}/\Gamma_{j}$ is a sphere and that the identity map $i:\R^{2}\to\R^{2}$ induces a quotient map $\overline{i}:S^{2}_{1}\to S^{2}_{2}$. This quotient map is a branched covering map of degree $d=[\Lambda_{2}:\Lambda_{1}]=9$. Let $\Phi:(\R^2,\Lambda_{2})\to(\R^2,\Lambda_{1})$ be the affine map $\Phi(x,y)=(3x,3y)$. The map $\Phi$ induces a homeomorphism $\phi$ from $S^{2}_{2}$ to $S^{2}_{1}$. Then we have the following commutative diagram:
$$\xymatrix{
\R^{2} \ar[d]_{/\Gamma_{1}} \ar[r]^{i} &\R^{2}\ar[d]_{/\Gamma_{2}}  \ar[r]^{\Phi} &\R^{2} \ar[d] \ar[d]_{/\Gamma_{1}}\\
S^{2}_{1} \ar[r]^{\overline{i}} & S^{2}_{2}\ar[r]^{\phi} &S^{2}_{1}}
$$
Below is a fundamental domain for the action of $\Gamma_{1}$ on $\R^{2}$. The red dots are elements of $\Lambda_{2}$. The lower left corner is $(0,0)$. Points are labeled (bold) by their images in $S_{1}^{2}$ under the map $/\Gamma_{1}$. The map $g:=\phi\circ\overline{i}$ is a Euclidean map with postcritical set $P_{g}=\Lambda_{1}/\Gamma_{1}=\{\text{a,b,c,d}\}$. Let $T_{1}:=\R^{2}/2\Lambda_{1}$ and let $p_{1}:T_{1}\to S_{1}^{2}$ be the map defined by $(x,y)+2\Lambda_{1}\mapsto/\Gamma_{1}(x,y)$. We identify $A$ with $\Lambda_{2}/2\Lambda_{1}$ so that the subset $H=\{\pm(0,2)+2\Lambda_{1},\pm(2,2)+2\Lambda_{1},\pm(2,4)+2\Lambda_{1},\pm(2,0)+2\Lambda_{1}\}$ is nonseparating in $\Lambda_{2}/2\Lambda_{1}$.
\begin{figure}[htb]
\center{\includegraphics[width=4.4in]{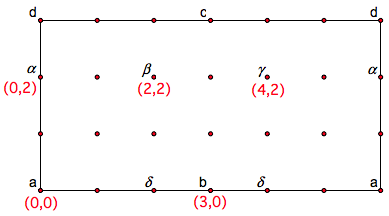}}
\caption[Fundamental domain for the action of $\Gamma_{1}$ on $\R^2$.]{\label{keywordforfigure}
{A fundamental domain for the action of $\Gamma_{1}$ on $\R^2$.}}
\end{figure}

Let $h:S_{1}^2\to S_{1}^2$ be an orientation-preserving homeomorphism so that $h(\text{a})=\delta$, $h(\text{b})=\beta$, $h(\text{c})=\alpha$, $h(\text{d})=\gamma$. Let $f:=h\circ g$. Then $P_{f}=\{\alpha,\beta,\gamma,\delta\}$ and $P_{f}=h(P_{g})=p_{1}(H)$, where $H$ is the nonseparating set (identification) contained in $T_{1}$. By Theorem \ref{thm:55}, the map $f$ is a NET map of degree 9 whose Teichm\"uller map is constant; however, because of 
Proposition \ref{prop:37}, this example does not satisfy the McMullen's constant conditions.\vspace{.1in}\\
\end{ex}
\section{Nonexistence results}
The following nonexistence results can be found in Section 10 of \cite{CFPP}.
\begin{thm}\label{thm:910} There does not exist a NET map with degree $2$ whose Teichm\"uller map is constant.
\end{thm} 

\begin{thm}\label{thm:911} Let $A$ be a finite Abelian group such that $A/2A\cong\mathbb{Z}/2\mathbb{Z}\oplus\mathbb{Z}/2\mathbb{Z}$ and $2A$ is a cyclic  group with odd order. Then $A$ does not contain a nonseparating subset.
\end{thm} 

\begin{thm}\label{thm:912} There does not exist a NET map with degree an odd square-free integer and constant Teichm\"uller map.
\end{thm} 

Our first nonexistence result is the following. 
\begin{thm}\label{thm:57} Let $A$ be a finite Abelian group generated by two elements such that $A/2A\cong\mathbb{Z}/2\mathbb{Z}\oplus\mathbb{Z}/2\mathbb{Z}$. If $|A|=4p^2$ with $p$ prime and $p\geq 5$, then $A$ does not contain a nonseparating subset.
\end{thm}
\begin{proof} The assumptions on the group $A$ imply that either $A\cong\mathbb{Z}/2\mathbb{Z}\oplus\mathbb{Z}/2p^2\mathbb{Z}$ or $A\cong\mathbb{Z}/2p\mathbb{Z}\oplus\mathbb{Z}/2p\mathbb{Z}$. If $A\cong\mathbb{Z}/2\mathbb{Z}\oplus\mathbb{Z}/2p^2\mathbb{Z}$, then $2A\cong\mathbb{Z}/p^2\mathbb{Z}$ which is a cyclic group with odd order. By Theorem \ref{thm:911}, $A$ does not contain a nonseparating set. So we may, and do, assume that $A=\mathbb{Z}/2\mathbb{Z}\oplus\mathbb{Z}/2\mathbb{Z}\oplus\mathbb{Z}/p\mathbb{Z}\oplus\mathbb{Z}/p\mathbb{Z}$. The rest of the proof is by contradiction. Suppose that $A$ contains a nonseparating subset $H=H_{1}\coprod H_{2}\coprod H_{3}\coprod H_{4}$. Here each $H_{i}$ has the form $H_{i}=\{\pm h_{i}\}$. Then either $\mathbb{Z}/p\mathbb{Z}\oplus\mathbb{Z}/p\mathbb{Z}\subseteq\langle H\rangle$ or $(\mathbb{Z}/p\mathbb{Z}\oplus\mathbb{Z}/p\mathbb{Z})\cap(A\setminus \langle H\rangle)\neq\emptyset$.

\textbf{Case 1.} $\mathbb{Z}/p\mathbb{Z}\oplus\mathbb{Z}/p\mathbb{Z}\subseteq\langle H\rangle$. There are in total $p+1$ subgroups of order $p$. Let $\psi:A\to\mathbb{Z}/2\mathbb{Z}\oplus\mathbb{Z}/2\mathbb{Z}$ and $\phi:A\to\mathbb{Z}/p\mathbb{Z}\oplus\mathbb{Z}/p\mathbb{Z}$ be the canonical projections defined by $\psi(a,b,c,d)=(a,b)$ and $\phi(a,b,c,d)=(c,d)$ respectively and consider the set of differences $D:=\{\phi(h_{i})\pm\phi(h_{j}): i,j\in\{1,2,3,4\}\ \text{and}\ i<j\}$. We begin by proving the theorem under the assumption that there exists a subgroup $G$ of $\zp$ with order $p$ such that $G\cap D\subseteq\{0\}$.

Suppose there exists a subgroup $G$ of $\zp$ with order $p$ such that $G\cap D\subseteq\{0\}$. It is clear that $(\Z/p\Z\oplus\Z/p\Z)/G$ is cyclic. By Lemma \ref{lem:a6}, we may, and do, assume that $\langle\phi(h_{1}),\phi(h_{2})\rangle=\Z/p\Z\oplus\Z/p\Z$ and  $h_{2}$, $h_{3}$ and $h_{4}$ all differ by an element of order 2. Choose $B$ so that it contains $h_{2}$. Choose $a$ arbitrarily such that $a+B$ generates $A/B$.  Then $h_{2}+B$, $h_{3}+B$, $h_{4}+B$ lie in the subgroup of order 2 in $A/B$.  They are not all equal because $B$ does not contain three elements of order 2.  So either $c_{1}=c_{2}=0$ and $c_{4}=p$ or $c_{1}=0$ and $c_{3}=c_{4}=p$.  Since $h_{1}+B$ does not have order 2 in $A/B$, we finally have a contradiction.

Now the argument separates into cases. In every case we obtain either a contradiction or a subgroup $G$ of $\zp$ with order $p$ such that $G\cap D\subseteq\{0\}$. By the above, this suffices to prove the theorem.
Let $\psi$ and $\phi$ as above and for each $i\in\{1,2,3,4\}$ set $G_{i}:=\langle\phi(h_{i})\rangle$.

\textbf{Subcase 1} Two of the elements $\phi(h_{1}), \phi(h_{2}), \phi(h_{3}), \phi(h_{4})$ are $0$. Then $D$ contains at most five elements. Since $p$ is prime and $p\geq5$, $\zp$ contains at least six subgroups with order $p$, so there exists a subgroup $G$ of $\zp$ with order $p$ such that $G\cap D\subseteq\{0\}$. This handles the subcase in which two of the elements $\phi(h_{1}), \phi(h_{2}), \phi(h_{3}), \phi(h_{4})$ are $0$.

\textbf{Subcase 2} One of the elements $\phi(h_{1}), \phi(h_{2}), \phi(h_{3}), \phi(h_{4})$ is $0$. Without loss of generality $\phi(h_{1})=0$. By translating $H$ by an element of order 2 if necessary as in Lemma \ref{lem:55}, we may assume that $h_{1}=0$. Now we choose $B$ so that it contains $h_{2}$. Then $c_{1}=c_{2}=0$. So $c_{3}=0$. Without loss of generality $h_{3}\in B$. Next choose $B'$ so that it contains $h_{4}$. Then $c'_{1}=c'_{2}=0$. So $c'_{3}=0$. Without loss of generality $h_{3}\in B'$. If $\phi(h_{3})=0$, then subcase 2 reduces to subcase 1. If $\phi(h_{3})\neq0$, then $\phi(h_{1}), \phi(h_{2}), \phi(h_{3}), \phi(h_{4})\in\langle h_{3}\rangle$. This contradicts the assumption that $\mathbb{Z}/p\mathbb{Z}\oplus\mathbb{Z}/p\mathbb{Z}\subseteq\langle H\rangle$. This handles the subcase in which one of the elements $\phi(h_{1}), \phi(h_{2}), \phi(h_{3}), \phi(h_{4})$ is $0$.

\textbf{Subcase 3} $G_{1}=G_{2}=G_{3}\neq\{0\}$. If $p=5$, then the elements $\pm\phi(h_{1}),  \pm\phi(h_{2})$ and $\pm\phi(h_{3})$ are not distinct. Hence $D\setminus G_{1}$ contains at most 4 elements. So there exists a subgroup $G$ of $\mathbb{Z}/5\mathbb{Z}\oplus\mathbb{Z}/5\mathbb{Z}$ with order $5$ such that $G\cap D\subseteq\{0\}$. If $p\geq7$, then $\zp$ contains at least 8 distinct subgroups of order $p$. Since $D\setminus G_{1}$ contains at most 6 elements, there exists a subgroup $G$ of $\zp$ with order $p$ such that $G\cap D\subseteq\{0\}$.

\textbf{Subcase 4} $G_{1}=G_{2}\neq\{0\}$. By translating $H$ by an element of order 2 if necessary, we may assume that $h_{1}$ has order $p$. Then $h_{1}$ and $h_{2} $ are both contained in a subgroup $B$ of $A$ with order $2p$. So $c_{1}=c_{2}=0$. So $c_{3}=0$. So either $G_{3}=G_{1}$ or $G_{4}=G_{1}$. Thus subcase 4 reduces to subcase 3.

\textbf{Subcase 5} $G_{i}\cap G_{j}=\{0\}$ for $i\neq j$. By subcase 2 we may assume that none of the elements $\phi(h_{1}), \phi(h_{2}), \phi(h_{3}), \phi(h_{4})$ is zero. Suppose for every $i\in\{1,2,3,4\}$ that there are three choices of indices $j$ and $k$ with $j<k$ and a sign such that the elements $\phi(h_{j})\pm\phi(h_{k})$ is in $G_{i}$. Then every element of $D$ is in $G_{1}\cup G_{2}\cup G_{3}\cup G_{4}$. So there exists a subgroup $G$ of $\zp$ with order $p$ such that $G\cap D\subseteq\{0\}$.

So we may assume that there are not three choices of indices $j$ and $k$ with $j<k$ and a sign such that the element $\phi(h_{j})\pm\phi(h_{k})$ is in $G_{1}$. By translating $H$ by an element of order 2 if necessary, we may assume that $h_{1}$ has order $p$.

In this paragraph we assume that $h_{1}, h_{2}, h_{3}, h_{4}$ all have order $p$ and obtain a contradiction. For this, let $B$ be a subgroup of $A$ with order $2p$ which contains $h_{1}$. Since $G_{1}\cap G_{2}=\{0\}$ and $h_{2}$ has order $p$, we may choose $a\in A$ such that $a+B$ generates $A/B$ and $h_{2}+B=2a+B$. We have that $c_{1}=0$. Furthermore, $c_{2}\neq0$ because $G_{1}\cap G_{j}=\{0\}$ for $j\in\{2,3,4\}$. In addition, $c_{2}\neq1$ because $h_{i}$ has order $p$ for every $i$. So $c_{2}=2$ because $h_{2}+B=2a+B$. So $c_{3}=2$. Without loss of generality $h_{2}+B=h_{3}+B$. Hence $G_{1}$ contains $\phi(h_{2})-\phi(h_{3})$. Now we repeat this argument with $h_{4}$ instead of $h_{2}$. We conclude that $G_{1}$ contains one of the elements $\phi(h_{4})\pm\phi(h_{2})$ or $\phi(h_{4})\pm\phi(h_{3})$. It follows that there are two and hence three choices of $j$ and $k$ with $j<k$ and signs such that $G_{1}$ contains $\phi(h_{j})\pm\phi(h_{k})$. This contradiction shows that $h_{1}, h_{2}, h_{3}$ and $h_{4}$ do not  all have order $p$. 

So we may assume that $h_{2}$ has order $2p$. Equivalently, $\psi(h_{2})\neq0$. There are two subgroups $B$ and $B'$ of $A$ with order $2p$ which contain $h_{1}$ but not $\psi(h_{2})$. Then $a=h_{2}$ is an element of $A$ such that $a+B$ generates $A/B$. Now $c_{1}=0$ and $c_{2}=1$. So $c_{3}=1$. Without loss of generality $h_{2}+B=h_{3}+B$. So $h_{2}-h_{3}\in B$. Similarly, $h_{2}\pm h_{k}\in B'$ for some $k\in\{3,4\}$. If $h_{2}+h_{3}\in B'$, then $\phi(h_{2})\pm\phi(h_{3})\in G_{1}$, whence $\phi(h_{2}), \phi(h_{3})\in G_{1}$. This is impossible because $G_{i}\cap G_{j}=\{0\}$ for $i\neq j$. If $h_{2}\pm h_{4}\in B'$, then as in the previous paragraph it follows that there are three choices of $j$ and $k$ with $j<k$ and signs such that $G_{1}$ contains $\phi(h_{j})\pm\phi(h_{k})$, which is impossible. So $h_{2}-h_{3}\in B\cap B'=G_{1}$. It follows that $\psi(h_{2})=\psi(h_{3})\neq0$. If $\psi(h_{4})\neq0$, then we repeat this argument with $h_{4}$ instead $h_{2}$ and find that there are three choices of $j$ and $k$ with $j<k$ and signs such that $G_{1}$ contains $\phi(h_{j})\pm\phi(h_{k})$, which is impossible.

We are left with the case in which $\psi(h_{1})=\psi(h_{4})=0$ and $\psi(h_{2})=\psi(h_{3})\neq0$. Now we choose $B$ to be the subgroup of $A$ with order $2p$ such that $\psi(B)$ does not contain $\psi(h_{2})$ and $\phi(B)$ does not contain any of the four elements $\phi(h_{1})\pm\phi(h_{4})$ or $\phi(h_{2})\pm\phi(h_{3})$. Regardless of how a generator of $A/B$ is chosen, $c_{2}=c_{3}$. So $B$ contains an element of the form $h_{i}\pm h_{j}$. Considering $\psi(B)$ shows that either $i, j\in\{1,4\}$ or $i,j\in\{2,3\}$. Considering $\phi(B)$ now yields a contradiction. 

This completes the subcase in which $G_{i}\cap G_{j}=\{0\}$ for $i\neq j$ and therefore this completes case 1.

\textbf{Case 2.} $(\mathbb{Z}/p\mathbb{Z}\oplus\mathbb{Z}/p\mathbb{Z})\cap(A\setminus \langle H\rangle)\neq\emptyset$. This means that $\langle H\rangle$ does not contain a copy of $\mathbb{Z}/p\mathbb{Z}\oplus\mathbb{Z}/p\mathbb{Z}$. Then, $H\subset\langle H\rangle\subseteq A'$, where $A'$ is a subgroup of $A$ isomorphic to $\mathbb{Z}/2\mathbb{Z}\oplus\mathbb{Z}/2\mathbb{Z}\oplus\mathbb{Z}/p\mathbb{Z}$. By Lemma \ref{lem:a05}, $H$ is a nonseparating set for $A'$. This contradicts Theorem \ref{thm:911}. 

This completes case 2 and therefore this completes the proof of the theorem.
\end{proof}

\begin{cor}\label{cor:58}
For any prime $p\geq5$ there does not exist a NET map with degree $p^2$ whose Teichm\"uller map is constant. 
\end{cor}
\begin{cor}
Let $p$ be a prime integer with $p\geq5$. Then $A=\Z/p\Z\oplus\Z/p\Z$ does not contain a nonseparating set.
\end{cor}

Our second nonexistence result is the following. 

\begin{thm}\label{thm:59} Let $A$ be a finite Abelian group generated by two elements such that $A/2A\cong\mathbb{Z}/2\mathbb{Z}\oplus\mathbb{Z}/2\mathbb{Z}$. If $|A|=4p^3$ with $p$ prime and $p\geq 7$, then $A$ does not contain a nonseparating subset.
\end{thm}
\begin{proof} The assumptions on the group $A$ imply that either $A\cong\mathbb{Z}/2\mathbb{Z}\oplus\mathbb{Z}/2p^3\mathbb{Z}$ or $A\cong\mathbb{Z}/2p\mathbb{Z}\oplus\mathbb{Z}/2p^2\mathbb{Z}$.
If $A\cong\mathbb{Z}/2\mathbb{Z}\oplus\mathbb{Z}/2p^3\mathbb{Z}$, then $2A\cong\mathbb{Z}/p^3\mathbb{Z}$ which is a cyclic group with odd order. By Theorem \ref{thm:911}, $A$ does not contain a nonseparating set. So we may, and do, assume that $A=\mathbb{Z}/2\mathbb{Z}\oplus\mathbb{Z}/2\mathbb{Z}\oplus\mathbb{Z}/p\mathbb{Z}\oplus\mathbb{Z}/p^2\mathbb{Z}$. The rest of the proof is by contradiction. Suppose that $A$ contains a nonseparating subset $H=H_{1}\coprod H_{2}\coprod H_{3}\coprod H_{4}$. Here each $H_{i}$ has the form $H_{i}=\{\pm h_{i}\}$. Then either $\mathbb{Z}/p\mathbb{Z}\oplus\mathbb{Z}/p^2\mathbb{Z}\subseteq\langle H\rangle$
 or $(\mathbb{Z}/p\mathbb{Z}\oplus\mathbb{Z}/p^2\mathbb{Z})\cap(A\setminus \langle H\rangle)\neq\emptyset$.
 
\textbf{Case 1.} $\mathbb{Z}/p\mathbb{Z}\oplus\mathbb{Z}/p^2\mathbb{Z}\subseteq\langle H\rangle$. The proof is similar to the proof of Case 1 of Theorem \ref{thm:57}. Let $\phi:A\to\mathbb{Z}/p\mathbb{Z}\oplus\mathbb{Z}/p^2\mathbb{Z}$ be the canonical projection and consider the set of differences $D:=\{\phi(h_{i})\pm\phi(h_{j}): i,j\in\{1,2,3,4\}\ \text{and}\ i<j\}$. We begin by proving the theorem under the assumption that there exists a cyclic subgroup $G$ of $\Z/p\Z\oplus\Z/p^2\Z$ such that $G\cap D\subseteq\{0\}$ and $(\Z/p\Z\oplus\Z/p^2\Z)/G$ is cyclic.

Suppose there exists a cyclic subgroup $G$ of $\Z/p\Z\oplus\Z/p^2\Z$ such that $G\cap D\subseteq\{0\}$ and $(\Z/p\Z\oplus\Z/p^2\Z)/G$ is cyclic. By Lemma \ref{lem:a6}, we may, and do, assume that $\langle\phi(h_{1}),\phi(h_{2})\rangle=\Z/p\Z\oplus\Z/p^2\Z$ and  $h_{2}$, $h_{3}$ and $h_{4}$ all differ by an element of order 2. Set $\tilde{A}:=\mathbb{Z}/2\mathbb{Z}\oplus\mathbb{Z}/2\mathbb{Z}\oplus\Z/p^2\Z\oplus\Z/p^2\Z$ and let $i_{c}:A\to\tilde{A}$ be the canonical monomorphism defined by $i_{c}(x,y,z,t)=(x,y,pz,t)$. By Lemma \ref{lem:56}, $i_{c}(H)$ is a nonseparating subset of $\tilde{A}$. Now, set $\mu:=i_{c}(\phi(h_{1}))$ and $\nu:=i_{c}(\phi(h_{2}))$. Since $\langle\phi(h_{1}),\phi(h_{2})\rangle=\Z/p\Z\oplus\Z/p^2\Z$, $\langle \mu,\nu\rangle=i_{c}(\Z/p\Z\oplus\Z/p^2\Z)=\langle p\rangle\oplus\Z/p^2\Z$. Furthermore, $i_{c}(h_{2})$, $i_{c}(h_{3})$ and $i_{c}(h_{4})$ all differ by an element of order 2. Since $\nu\in\Z/p^2\Z\oplus\Z/p^2\Z$ there exists a cyclic subgroup $\tilde{B}$ of $\tilde{A}$ of order $2p^2$ such that $i_{c}(h_{2})\in\tilde{B}$ and $\tilde{A}/\tilde{B}$ is cyclic. Then $i_{c}(h_{2})+\tilde{B}$, $i_{c}(h_{3})+\tilde{B}$, $i_{c}(h_{4})+\tilde{B}$ lie in the subgroup of order 2 in $\tilde{A}/\tilde{B}$. They are not all equal because $\tilde{B}$ does not contain three elements of order 2.  So either $c_{1}=c_{2}=0$ and $c_{4}=p^2$ or $c_{1}=0$ and $c_{3}=c_{4}=p^2$. Thus,  $i_{c}(h_{1})+\tilde{B}$ must have order 2 in $\tilde{A}/\tilde{B}$. Hence $2\mu\in\tilde{B}$ and so $\mu\in\tilde{B}$. Then $\langle\mu,\nu\rangle\subseteq\tilde{B}$, which is impossible because $\tilde{B}$ is cyclic and $\langle\mu,\nu\rangle=\langle p\rangle\oplus\Z/p^2\Z$.

Now the argument separates into cases. In every case we obtain a cyclic subgroup $G$ of $\Z/p\Z\oplus\langle p\rangle$ such that $G\cap D\subseteq\{0\}$ and $(\Z/p\Z\oplus\Z/p^2\Z)/G$ is cyclic. By the above, this suffices to prove the theorem.

First of all, note that there exists $h_{i}\in H$ so that $\phi(h_{i})\notin\Z/p\Z\oplus\langle p\rangle$. Otherwise, by Lemma \ref{lem:a05}, $H$ is a nonseparating subset of $\Z/2\Z\oplus\Z/2\Z\oplus Z/p\Z\oplus\langle p\rangle$, which contradicts Theorem \ref{thm:57}. So without loss of generality, we assume that $\phi(h_{1})\notin\Z/p\Z\oplus\langle p\rangle$. Also note that $\Z/p\Z\oplus\langle p\rangle$ contains $p$ distinct cyclic subgroups $G_{0},G_{1},\cdots, G_{p-1}$ such that $(\Z/p\Z\oplus\Z/p^2\Z)/G_{i}$ is cyclic for $i\in\{0,1,\cdots,p-1\}$. For instance, consider $G_{0}=\langle (1,0)\rangle$,  $G_{i}=\langle (i,p)\rangle$, where $i\in\{1,2,\cdots, p-1\}$.

\textbf{Subcase 1} $\phi(h_{i})\notin\Z/p\Z\oplus\langle p\rangle$ for all $i\in\{2,3,4\}$. Then, for any choice of indeces $j$ and $k$ with $j<k$, $\phi(h_{j})+\phi(h_{k})$ and $\phi(h_{j})-\phi(h_{k})$ cannot be both elements of $\Z/p\Z\oplus\langle p\rangle$. So $\Z/p\Z\oplus\langle p\rangle$ contains at most six elements of $D$. 

\textbf{Subcase 2} $\phi(h_{2}),\phi(h_{3})\notin\Z/p\Z\oplus\langle p\rangle$ but $\phi(h_{4})\in\Z/p\Z\oplus\langle p\rangle$. Then, for any choice of indeces $j$ and $k$ with $1\leq j<k\leq3$, $\phi(h_{j})+\phi(h_{k})$ and $\phi(h_{j})-\phi(h_{k})$ cannot be both elements of $\Z/p\Z\oplus\langle p\rangle$. Also, $\phi(h_{i})\pm\phi(h_{4})\notin\Z/p\Z\oplus\langle p\rangle$ for $i\in\{1,2,3\}$. So $\Z/p\Z\oplus\langle p\rangle$ contains at most three elements of $D$. 

\textbf{Subcase 3} $\phi(h_{2})\notin\Z/p\Z\oplus\langle p\rangle$ but $\phi(h_{3}),\phi(h_{4})\in\Z/p\Z\oplus\langle p\rangle$. Then $\phi(h_{1})\pm \phi(h_{2})$ cannot be both elements of $\Z/p\Z\oplus\langle p\rangle$ and $\phi(h_{i})\pm\phi(h_{j})\notin\Z/p\Z\oplus\langle p\rangle$ for $i\in\{1,2\}$ and $j\in\{3,4\}$. So $\Z/p\Z\oplus\langle p\rangle$ contains at most three elements of $D$. 

\textbf{Subcase 4} $\phi(h_{2}),\phi(h_{3}),\phi(h_{4})\in\Z/p\Z\oplus\langle p\rangle$. It follows that there are at most six elements of $D$ in $\Z/p\Z\oplus\langle p\rangle$. 

In any case, since $p$ is prime and $p\geq7$, there exists a cyclic subgroup $G$ of $\Z/p\Z\oplus\langle p\rangle$ such that $G\cap D\subseteq\{0\}$ and $(\Z/p\Z\oplus\Z/p^2\Z)/G$ is cyclic. This completes case 1.

\textbf{Case 2.} $(\mathbb{Z}/p\mathbb{Z}\oplus\mathbb{Z}/p^2\mathbb{Z})\cap(A\setminus \langle H\rangle)\neq\emptyset$. This means that $\langle H\rangle$ does not contain a copy of $\mathbb{Z}/p\mathbb{Z}\oplus\mathbb{Z}/p^2\mathbb{Z}$. Then, $H\subset\langle H\rangle\subseteq A'$, where $A'$ is a subgroup of $A$ generated by 2 elements whose order is $4p^2$. By Lemma \ref{lem:a05}, $H$ is a nonseparating set for $A'$. This contradicts Theorem \ref{thm:57}. 

This completes case 2 and therefore this completes the proof of the theorem.
\end{proof}

\begin{cor}\label{cor:60}
For any prime $p\geq7$ there does not exist a NET map with degree $p^3$ whose Teichm\"uller map is constant.
\end{cor}

Our main theorem is a nonexistence result. We begin with the following lemma.

\begin{lem}\label{lem:60} Let $N=p_{1}^{k_{1}}\cdots p_{n}^{k_{n}}$ with $p_{i}\geq13$. Let $A=\Z/a\Z\oplus\Z/b\Z$ such that $a|b$, $N=ab$ and $a>1$. If $D\subseteq A$ so that $\#D\leq12$, then there exists a cyclic subgroup $G$ of $A$ such that $G\cap D\subseteq\{0\}$ and $A/G$ is cyclic.\end{lem}
\begin{proof}
Since $a|b$ and $N=ab$, then $a=p_{1}^{s_{1}}\cdots p_{n}^{s_{n}}$ and $b=p_{1}^{k_{1}-s_{1}}\cdots p_{n}^{k_{n}-s_{n}}$, where $0\leq 2s_{i}\leq k_{i}$. Let $I=\{1,\cdots,n\}$ and define $I_{1}:=\{i\in I: s_{i}=0\}$ and $I_{2}:=I\setminus I_{1}$. Since $a>1$, $I_{2}$ cannot be empty. Then $A\cong C\oplus P$, where $C=\ds\oplus_{i\in I_{1}}\Z/p_{i}^{k_{i}}\Z$ and $P=\oplus_{i\in I_{2}}(\Z/p_{i}^{s_{i}}\Z\oplus\Z/p_{i}^{k_{i}-s_{i}}\Z)$. Without loss of generality we may assume that $A=C\oplus P$. Now, for each $i\in I_{2}$ let $\phi_{i}:A\to \Z/p_{i}^{s_{i}}\Z\oplus\Z/p_{i}^{k_{i}-s_{i}}\Z$ be the canonical projection. Since $\#\phi_{i}(D)\leq12$ and $p_{i}\geq13$, there exists a cyclic subgroup $G_{i}$ of $\Z/p_{i}^{s_{i}}\Z\oplus\Z/p_{i}^{k_{i}-s_{i}}\Z$ such that $G_{i}\cap\phi_{i}(D)\subseteq\{0\}$ and $\big(\Z/p_{i}^{s_{i}}\Z\oplus\Z/p_{i}^{k_{i}-s_{i}}\Z\big) / G_{i}$ is cyclic. The subgroup $G:=\oplus_{i\in I_{2}}G_{i}$ satisfies the conclusion.
\end{proof}

\begin{thm}\label{thm:main} Let $A$ be a finite Abelian group generated by two elements such that $A/2A\cong\mathbb{Z}/2\mathbb{Z}\oplus\mathbb{Z}/2\mathbb{Z}$. If 
$|A|=4p_{1}^{k_{1}}p_{2}^{k_{2}}\cdots p_{n}^{k_{n}}$ with  $p_{i}$ prime, $p_{i}\geq 13$ and $|k|=k_{1}+k_{2}+\cdots +k_{n}\geq1$, then $A$ does not contain a nonseparating subset.
\end{thm}
\begin{proof} We proceed by induction on $|k|$. If $|k|=1$ then $A\cong\mathbb{Z}/2\mathbb{Z}\oplus\Z/2p\Z$. Since $2A$ is cyclic with odd order, by Theorem \ref{thm:911}, the conclusion follows. 
Now, suppose the conclusion holds for any $|k|\in\{1,\cdots,m-1\}$ and assume that $|A|=4p_{1}^{k_{1}}p_{2}^{k_{2}}\cdots p_{n}^{k_{n}},$ where $k_{1}+\cdots +k_{n}=m$. Then there are two positive odd integers $a$ and $b$ such that $A\cong\mathbb{Z}/2a\mathbb{Z}\oplus\mathbb{Z}/2b\mathbb{Z}$, $|A|=4ab$, and $a$ divides $b$. If $a=1$, then $2A$ is a cyclic group with odd order. By Theorem \ref{thm:911}, the conclusion follows. Now, assume that $A=\mathbb{Z}/2\mathbb{Z}\oplus\mathbb{Z}/2\mathbb{Z}\oplus\Z/a\Z\oplus\Z/b\Z$ with $a>1$ and proceed by contradiction. Suppose that $A$ contains a nonseparating subset $H=H_{1}\coprod H_{2}\coprod H_{3}\coprod H_{4}$, where each $H_{i}=\{\pm h_{i}\}$. Then, either $\Z/a\Z\oplus\Z/b\Z\subseteq\langle H\rangle$ or $(\Z/a\Z\oplus\Z/b\Z)\cap(A\setminus \langle H\rangle)\neq\emptyset$.

\textbf{CASE 1.}  $\Z/a\Z\oplus\Z/b\Z\subseteq\langle H\rangle$. Let $\phi:A\to\Z/a\Z\oplus\Z/b\Z$ be the canonical projection and let $D:=\{\phi(h_{i})\pm\phi(h_{j}): i,j\in\{1,2,3,4\}\ \text{and}\ i<j\}$. The cardinality of $D$ is at most $12$. By Lemma \ref{lem:60}, there exists a cyclic subgroup $G$ of $\Z/a\Z\oplus\Z/b\Z$ such that $G\cap D\subseteq\{0\}$ and $(\Z/a\Z\oplus\Z/b\Z)/G$ is cyclic. Then, by Lemma \ref{lem:a6}, we may, and do, assume that $\langle\phi(h_{1}),\phi(h_{2})\rangle=\Z/a\Z\oplus\Z/b\Z$ and  $h_{2}$, $h_{3}$ and $h_{4}$ all differ by an element of order 2. 
 
If $a=b$, there exists a cyclic subgroup $B$ of $A$ of order $2b$ such that $h_{2}\in B$ and $A/B$ is cyclic. Since $h_{2}$, $h_{3}$ and $h_{4}$ all differ by an element of order 2, then $h_{2}+B$, $h_{3}+B$, $h_{4}+B$ lie in the subgroup of order 2 in $A/B$. They are not all equal because $B$ does not contain three elements of order 2.  So either $c_{1}=c_{2}=0$ and $c_{4}=b$ or $c_{1}=0$ and $c_{3}=c_{4}=b$. Thus, $h_{1}+B$ must have order 2 in $A/B$. Hence $2\phi(h_{1})\in B$ and so $\phi(h_{1})=((b+1)/2)(2\phi(h_{1}))\in B$. Therefore $\langle\phi(h_{1}),\phi(h_{2})\rangle\subseteq B$. This is a contradiction because $B$ is cyclic and $\langle\phi(h_{1}),\phi(h_{2})\rangle=\Z/a\Z\oplus\Z/b\Z$.

If $a<b$, set $\tilde{A}:=\mathbb{Z}/2\mathbb{Z}\oplus\mathbb{Z}/2\mathbb{Z}\oplus\Z/b\Z\oplus\Z/b\Z$ and let $i_{c}$  be the canonical monomorphism $i_{c}:A\to\tilde{A}$ defined by $i_{c}(x,y,z,t)=(x,y,(b/a)z,t)$. By Lemma \ref{lem:56} $i_{c}(H)$ is a nonseparating subset of $\tilde{A}$. Now, set $\mu:=i_{c}(\phi(h_{1}))$ and $\nu:=i_{c}(\phi(h_{2}))$. Since $\langle\phi(h_{1}),\phi(h_{2})\rangle=\Z/a\Z\oplus\Z/b\Z$, then $\langle \mu,\nu\rangle=i_{c}(\Z/a\Z\oplus\Z/b\Z)=\langle b/a\rangle\oplus\Z/b\Z$. Furthermore, $i_{c}(h_{2})$, $i_{c}(h_{3})$ and $i_{c}(h_{4})$ all differ by an element of order 2. Since $\nu\in\Z/b\Z\oplus\Z/b\Z$ there exists a cyclic subgroup $\tilde{B}$ of $\tilde{A}$ of order $2b$ such that $i_{c}(h_{2})\in\tilde{B}$ and $\tilde{A}/\tilde{B}$ is cyclic. Then $i_{c}(h_{2})+\tilde{B}$, $i_{c}(h_{3})+\tilde{B}$, $i_{c}(h_{4})+\tilde{B}$ lie in the subgroup of order 2 in $\tilde{A}/\tilde{B}$. They are not all equal because $\tilde{B}$ does not contain three elements of order 2.  So either $c_{1}=c_{2}=0$ and $c_{4}=b$ or $c_{1}=0$ and $c_{3}=c_{4}=b$. Thus,  $i_{c}(h_{1})+\tilde{B}$ must have order 2 in $\tilde{A}/\tilde{B}$. Hence $2\mu\in\tilde{B}$ and so $\mu=((b+1)/2)(2\mu)\in\tilde{B}$. Therefore $\langle\mu,\nu\rangle\subseteq\tilde{B}$. This is a contradiction because $\tilde{B}$ is cyclic and $\langle\mu,\nu\rangle=\langle b/a\rangle\oplus\Z/b\Z$.
 
\textbf{CASE 2.}   $(\Z/a\Z\oplus\Z/b\Z)\cap(A\setminus \langle H\rangle)\neq\emptyset$.
This means that $\langle H\rangle$ does not contain a copy of $\Z/a\Z\oplus\Z/b\Z$. Then, $H\subset\langle H\rangle\subseteq A'$ where $A'$ is a proper subgroup of $A$ whose order has the form $4r$. Obviously, $r$ divides $|A|/4$ and $r<|A|/4$. Since $A'$ is a finite Abelian group generated by 2 elements, we can apply the inductive hypothesis to $A'$ and conclude that $A'$ does not contain a nonseparating subset. However, this contradicts Lemma \ref{lem:a05}.

This proves Theorem \ref{thm:main}.
\end{proof}\

\begin{cor} Let $n=p_{1}^{k_{1}}p_{2}^{k_{2}}\cdots p_{n}^{k_{n}}$ with  $p_{i}$ prime, $p_{i}\geq 13$. There does not exist a NET map with degree $n$ whose Teichm\"uller map is constant.
\end{cor}
\newpage
\section{On Thurston maps of degree 2}
Let $f:S^2\to S^{2}$ be a Thurston map and $P_{f}$ its postcritical set. Combining statements 1 and 4 of Theorem 5.1 of \cite{BEKP} implies that the pullback map $\Sigma_{f}$ is constant if and only if for every essential simple closed curve $\alpha$ in $S^{2}\setminus P_{f}$, every connected component of $f^{-1}(\alpha)$ is either trivial or peripheral in $S^{2}\setminus P_{f}$. We use this result to conclude that there does not exist a Thurston map of degree 2 with at least four postcritical points whose Teichm\"uller map is constant.

A Thurston map $f$ is a $topological\ polynomial$ if there exists a critical point $w$, such that $f^{-1}(w)=\{w\}$. If $f$ is a Thurston polynomial and $|P_{f}|>2$, there is a unique point $w$ such that $f^{-1}(w)=\{w\}$; we call this point $\infty$.  

 \begin{pro}\label{pro:B1}
 Let $f$ be a quadratic topological polynomial. If $|P_{f}|\geq4$, then the Teichm\"uller map $\Sigma_{f}$ cannot be constant.
 \end{pro}\begin{proof}
 If $|P_{f}|=4$ then $f$ is a NET map. By Theorem \ref{thm:910}, $\Sigma_{f}$ cannot be constant. From now on, assume that $|P_{f}|\geq5$. Since $\deg(f)=2,$ $f$ has two critical points; i.e. $\Omega_{f}=\{a,\infty\}$ where $a$ is some point in $S^2\setminus\{\infty\}$. We first assume that $z=a$ is a preperiodic critical point. Enumerate the finite postcritical points as $c_{1},c_{2},\cdots,c_{k}$ where $c_{j}=f^{\circ j}(a)$. The ramification portrait of $f$ is given by: 
 \begin{displaymath}
    \xymatrix{
    c_{0}=a \ar[r]^-{2} & c_{1}\ar[r]&\cdots \ar[r]& c_{i}\ar[r]&\cdots\ar[r]& c_{k}\ar@/^1pc/[ll]&  &\infty\ar@(dr,ur)[]_{2} }
\end{displaymath}\\
where $k=|P_{f}|-1$ and $i$ is some integer in $\{2,\cdots,k\}$. Note that $k\geq4$. Let $\gamma$ be an arc joining the points $\infty$ and $c_{i}$  such that $\gamma\cap P_{f}=\{c_{i},\infty\}$. Then $f^{-1}(\gamma)$ is the union of two arcs $\gamma_{1}$ and $\gamma_{2}$ so that $\gamma_{1}\cap\gamma_{2}=\{\infty\}$, $\gamma_{1}$ joins $\infty$ and $c_{k}$, and $\gamma_{2}$ joins $\infty$ and $c_{i-1}$. Now let $\alpha$ be the boundary of a small regular neighborhood of the arc $\gamma$. By continuity, we may take $\alpha$ to be a simple closed curve so that $\gamma$ is a core arc for $\alpha$ and $\alpha\cap P_{f}=\emptyset$. Then  $\gamma_{1}\cup\gamma_{2}$ is a \lq\lq core arc\rq\rq \ for $f^{-1}(\alpha)$. Figure \ref{keywordforfigure} illustrates $\alpha$ near $\gamma$ and their respective pre images. Thus, each connected  component of $S^{2}\setminus f^{-1}(\alpha)$ contains at least two postcritical points. Therefore, $\Sigma_{f}$ cannot be constant.

\begin{figure}[htb]
\center{\includegraphics[height=1.6in]{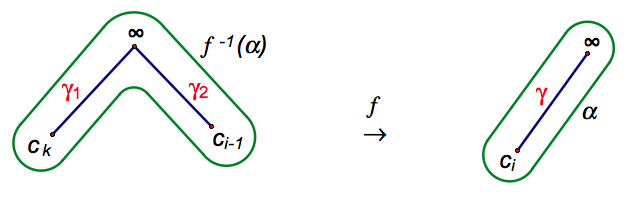}}
\caption[Preimage of a core arc under the topological polynomial $f$.]{\label{keywordforfigure}
{Small regular neighborhoods of the arc $\gamma$ and of its preimage}}
\end{figure}
 Now assume that $z=a$ is a periodic critical point. There is a unique point $c\in P_{f}$ so that $f(c)=a$. Note that $f$ maps $c$ to $a$ with degree $1$. Let $\gamma$ be an arc joining the points $\infty$ and $a$. Now let $\alpha$ be the boundary of a small regular neighborhood of the arc $\gamma$. Proceeding as in the preperiodic case, each connected  component of $S^{2}\setminus f^{-1}(\alpha)$ contains at least two postcritical points. Therefore, $\Sigma_{f}$ cannot be constant.
\end{proof}

S. Koch have provided an analytic proof of the preceding proposition. In Proposition 5.3 of \cite{K}, Koch shows that if $f$ is a bicritical topological polynomial, then $X:\mathcal{W}_{f}\to\mathcal{M}_{P_{f}}$ is inyective. So a moduli space map $g_{f}$ exists. This implies that $\Sigma_{f}$ cannot be constant. For further details, see Proposition 5.3  and Corollary 5.4 of \cite{K}.

\begin{teo}
Let $f$ be a Thurston map of degree 2.  If $|P_{f}|\geq4$, then $\Sigma_{f}$ cannot be constant.
\end{teo}
\begin{proof} We may assume that $f$ is not a quadratic topological polynomial. So $f$ has two critical points and neither of them is a fixed point. If $|P_{f}|=4$, then $f$ is a NET map  and by Theorem \ref{thm:910} conclusion follows. From now on assume that $|P_{f}|\geq5$. Let $\Omega_{f}=\{a,b\}$ be the set of critical points of $f$. Set $\mathcal{O}(a)=\{f^{\circ i}(a):i\in\mathbb{N}\}$  and $\mathcal{O}(b)=\{f^{\circ i}(b):i\in\mathbb{N}\}$ and analyze two cases.

Case I. $\mathcal{O}(a)\cap\mathcal{O}(b)=\emptyset$.\\
\textbf{1.} If $a\in P_{f}$, let $\gamma$ be an arc joining $a$ and $f(a)$ so that $\gamma\cap( P_{f}\setminus\{a,f(a)\})=\emptyset$. If $b\in P_{f}$, let $\gamma$ be an arc joining $b$ and $f(b)$ so that $\gamma\cap( P_{f}\setminus\{b,f(b)\})=\emptyset$.  Then, in any case, proceeding as in the quadratic topological polynomial case one sees that $\Sigma_{f}$ cannot be constant.\\
\textbf{2.} If $a\notin P_{f}$ and $b\notin P_{f}$,  then $a$ and $b$ are both preperiodic critical points. Then, either $\mathcal{O}(a)$ and $\mathcal{O}(b)$ contain fixed points or one of them, say $\mathcal{O}(a)$, contains no fixed point. If $\mathcal{O}(a)$ and $\mathcal{O}(b)$ contain fixed points, namely $a_{k}$ and $b_{r}$, take the preimage of a suitable core arc joining  $a_{k}$ and $b_{r}$ and proceed as in the quadratic topological polynomial case. Thus $\Sigma_{f}$ cannot be constant. If $\mathcal{O}(a)$ has no fixed point, there are two distinct points $p,q\in\mathcal{O}(a)$ so that $f(p)=f(q)$. Set $x:=f(p)$. Now, let $\gamma$ be an arc joining the points $x$ and $f(x)$ so that $\gamma\cap( P_{f}\setminus\{x,f(x)\})=\emptyset$. Then, proceeding as in the quadratic topological polynomial case, $\Sigma_{f}$ cannot be constant.

Case II. $\mathcal{O}(a)\cap\mathcal{O}(b)\neq\emptyset$.\\  
\textbf{1.} If $a\in P_{f}$ or  $b\in P_{f}$ proceed as in Case I to conclude that $\Sigma_{f}$ cannot be constant.\\
\textbf{2.} If $a\notin P_{f}$ and $b\notin P_{f}$. Set $k:=\min\{i\in\mathbb{N}: f^{\circ i}(b)\in\mathcal{O}(a)\}$. Since $b\notin P_{f}$ and $f$ maps $b$ to $f(b)$ with degree $2$, we have $k\geq2$. Set $x:=f^{\circ k}(b)$. Then $x=f(q)$ where $q=f^{\circ k-1}(b)$. By definition of $k$, $x\in\mathcal{O}(a)\setminus\{a,f(a)\}$. Thus, $x=f(p)$ for some $p\in\mathcal{O}(a)\cap P_{f}$. Due to the minimality of $k$, $p\neq q$. Then the ramification portrait of $f$ contains the directed subgraph  
\[\xymatrix{
&p \ar[rd]\\
& & x \ar[r]&f(x)\\
&q \ar[ru]\\ 
}\]
Let $\gamma$ be an arc joining the points $x$ and $f(x)$ so that $\gamma\cap( P_{f}\setminus\{x,f(x)\})=\emptyset$. Then, proceeding as in the quadratic topological polynomial case, $\Sigma_{f}$ cannot be constant. 
\end{proof}

The following is a slight generalization of Proposition \ref{pro:B1}.
 \begin{pro}
 Let $f$ be a topological polynomial of degree $n$ so that $|P_{f}|=m\geq4$. Suppose there exists $c\in P_{f}\setminus\{\infty\}$ such that $f^{-1}({c})$ contains no critical points. Let $k=|f^{-1}({c})\cap P_{f}|$. If $m-k\geq3$ then $\Sigma_{f}$ cannot be constant. 
 \end{pro}
 \begin{proof}
Since $c\in P_{f}\setminus\{\infty\}$ and $f^{-1}({c})$ contains no critical points, then $k>0$. Let $\gamma$ be an arc joining the points $\infty$ and $c$  such that $\gamma\cap( P_{f}\setminus\{c,\infty\})=\emptyset$. Then $f^{-1}(\gamma)$ is the union of $n$ arcs $\gamma_{1},\gamma_{2},\cdots,\gamma_{n} $ so that $\gamma_{1}\cap\gamma_{2}\cdots\cap\gamma_{n}=\{\infty\}$, and each $\gamma_{i}$ joins $\infty$ and some preimage of $c$. Now let $\alpha$ be the boundary of a small regular neighborhood of the curve $\gamma$. By continuity, we may take $\alpha$ to be a simple closed curve so that $\gamma$ is a core arc for $\alpha$ and $\alpha\cap( P_{f}\setminus\{c,\infty\})=\emptyset$. Then one connected component of $S^{2}\setminus f^{-1}(\alpha)$  contains the set $\{\infty\}\cup(f^{-1}({c})\cap P_{f})$; so this connected component of  $S^{2}\setminus f^{-1}(\alpha)$ contains exactly $k+1$ postcritical points. Thus the other connected component contains $m-(k+1)$ postcritical points. Since $m-(k+1)\geq2$ we conclude that $\Sigma_{f}$ cannot be constant.
\end{proof}

\maketitle\appendix
\section{Group Theory}
\begin{prop}\label{prop:a1} Let $G$ be a finite cyclic group of order $n$ and let $h$ be an element of order $m$ in $G$. Then there exists $g\in G$ such that $\langle g\rangle=G$ and $g^{n/m}=h$.
\end{prop}
\begin{proof}
Choose $a\in G$ such that $\langle a\rangle=G$. Then $\langle a^{n/m}\rangle=\langle h\rangle$, so there exists $r\in\n$ such that $a^{nr/m}=h$; of course $\gcd(r,m)=1$. Also if $d\in\n$,
then $a^{nd/m}=h$ if and only if $d\equiv r\mod m$. Since $n=m(n/m)$, analyze the following two cases.

\text{First Case.} Every prime divisor of $n/m$ is a divisor of $m$. Then $\gcd(r,n)=1$. Otherwise, there exists a prime number $p$ such that $p|r$ and $p|m(n/m)$. Hence $p|r$ and $p|m$, which is a contradiction because $\gcd(r,m)=1$. Therefore, $\gcd(r,n)=1$ and so $a^{r}$ generates $G$. Also $(a^{r})^{n/m}=h$, as required.

\text{Second Case.} Not every prime divisor of $n/m$ is a divisor of $m$. Let $q$ be the product of the primes which divide $n/m$ but do not divide $m$. Thus, $\gcd(m,q)=1$ and $\gcd(s,n)=1$ if and only if $\gcd(s,m)=\gcd(s,q)=1$. By the Chinese remainder theorem, we may choose $\tau\in\n$ such that $\tau\equiv r \mod m$ and $\tau\equiv 1\mod q$. 
Then $\gcd(\tau,n)=1$, so $a^{\tau}$ generates $G$. Also $(a^{\tau})^{n/m}=h$, as required.
\end{proof}

\begin{prop}\label{prop:a2} Every element of $\mathbb{Z}/n\mathbb{Z}\oplus\mathbb{Z}/n\mathbb{Z}$ is a multiple of a basis element.
\end{prop}
\begin{proof}
In fact, let $g=(\ov{x},\ov{y})\in\mathbb{Z}/n\mathbb{Z}\oplus\mathbb{Z}/n\mathbb{Z}$. If $\ov{x}=\ov{0}$ or $\ov{y}=\ov{0}$, the proposition follows. Assume $\ov{x}\neq\ov{0}$ and $\ov{y}\neq\ov{0}$, then $g=d(\ov{x/d},\ov{y/d})$ where $d=\gcd(x,y)$. Since $\gcd(x/d,y/d)=1$, $(\ov{x/d},\ov{y/d})$ is a basis element and the proposition follows. 
\end{proof}

\begin{prop}\label{prop:a3} Assume that $A=\mathbb{Z}/n\mathbb{Z}\oplus\mathbb{Z}/n\mathbb{Z}$. Let $A'$ be a subgroup of $A$, and let $B'$ be a cyclic subgroup of $A'$ so that $A'/B'$ is also cyclic.  Then there exists a cyclic subgroup $B$ of $A$ such that $A/B$ is cyclic and $A'\cap B = B'$.
\end{prop}

\begin{proof} Since $A'$ is a finite Abelian group generated by two elements, $A'$ is the internal direct sum of its Sylow subgroups. Then $A'=A'_{1}+ A'_{2}$ where $A'_{1}$ is the subgroup of $A'$ generated by its cyclic Sylow subgroups and $A'_{2}$ is the subgroup of $A'$ generated by its noncyclic Sylow subgroups. It is clear that $A'_{1}$ is cyclic. Let $|A'_{1}|=\alpha$ and $|A'_{2}|=\beta$. Then $\gcd(\alpha,\beta)=1$ and $\alpha\beta$ divides $n^2=p_{1}^{2s_{1}}\cdots p_{r}^{2s_{r}}$. Let $p_{i_{1}},\cdots,p_{i_{k}}$ be the distinct primes that divide $\alpha$. Now set
$$n_{1}=p_{i_{1}}^{s_{i_{1}}}\cdots p_{i_{k}}^{s_{i_{k}}}\ \ \ \text{and}\ \ \ n_{2}=n/n_{1}=p_{j_{1}}^{s_{j_{1}}}\cdots p_{j_{l}}^{s_{j_{l}}}.$$ 
It is not difficult to see that $\{i_{1},\ldots,i_{n}\}\coprod\{j_{1},\ldots j_{l}\}=\{1,\ldots,r\}$, $\alpha$ divides $n_{1}^2$, $\beta$ divides $n_{2}^2$ and  $\gcd(n_{1},n_{2})=1$. Since $n=n_{1}n_{2}$ and $\gcd(n_{1},n_{2})=1$,  then the group $A$ is isomorphic to $\mathbb{Z}/n_{1}\mathbb{Z}\oplus\mathbb{Z}/n_{1}\mathbb{Z}\oplus\mathbb{Z}/n_{2}\mathbb{Z}\oplus\mathbb{Z}/n_{2}\mathbb{Z}$. So there exist subgroups $A_{1}$ and $A_{2}$ of $A$ such that $A=A_{1}+ A_{2}$, where $A_{1}\cong\mathbb{Z}/n_{1}\mathbb{Z}\oplus\mathbb{Z}/n_{1}\mathbb{Z}$ and $A_{2}\cong\mathbb{Z}/n_{2}\mathbb{Z}\oplus\mathbb{Z}/n_{2}\mathbb{Z}$. Thus, $A=A_{1}+A_{2}$ with $|A_{1}|=n_{1}^2$ and $|A_{2}|=n_{2}^2$.

Let $B'$ be a cyclic subgroup of $A'=A'_{1}+ A'_{2}$ and let $g\in B'$ be a generator of $B'$. Then $g$ can be written uniquely as $g=b'_{1}+b'_{2}$ where $b'_{i}\in A'_{i}$. Since $o(b'_{i})$ divides $|A'_{i}|$ and $\gcd(|A'_{1}|,|A'_{2}|)=1$, we have that $\gcd(o(b'_{1}),o(b'_{2}))=1$. Then $o(g)=o(b'_{1})\cdot o(b'_{2})$ and so $B'=\left<b'_{1}\right>+\left<b'_{2}\right>$. Hence $B'$ is the internal direct sum of $B'_{1}=\left<b'_{1}\right>$ and $B'_{2}=\left<b'_{2}\right>$. It is clear that $B'_{i}$ is a subgroup of $A'_{i}$ for $i\in\{1,2\}$. On the other hand, $A'_{i}\subset A_{i}$ for $i\in\{1,2\}$. In fact, let $x\in A'_{1}$, then $x$ can be written uniquely as $x=a_{1}+a_{2}$ where $a_{i}\in A_{i}$. Then $0=\alpha x=\alpha a_{1}+\alpha a_{2},$ where $|A'_{1}|=\alpha$. Thus,  $-\alpha a_{1}=\alpha a_{2}\in A_{1}\cap A_{2}$ and so $\alpha a_{2}=0$. If $a_{2}\neq0$, then $o(a_{2})$ divides $\alpha$ which is impossible because $\gcd(\alpha,|A_{2}|)=1$. Therefore, $a_{2}=0$ and so $x\in A_{1}$. For the case $i=2$, proceed similarly. 

By Proposition \ref{prop:a2} , every element of $A_1$ is a multiple of a basis element. Let $v\in A_{1} $ be a basis element such that $\left< v \right>$ contains $A'_{1}$. Choose $w\in A_{1}$ so that $\{v,w\}$ is a basis for $A_{1}$. Let $\varphi:A_{1}\to\mathbb{Z}/n_{1}\mathbb{Z}\oplus\mathbb{Z}/n_{1}\mathbb{Z}$ be the isomorphism defined on these generators by $\varphi(v)=(1,0)$ and $\varphi(w)=(0,1)$. Let $k=o(\varphi(b'_{1}))$. Since $\varphi(b'_{1})\in\left<(1,0)\right>$, by proposition \ref{prop:a1}, there exists $g\in\mathbb{Z}/n_{1}\mathbb{Z}$ such that $\langle g\rangle=\mathbb{Z}/n_{1}\mathbb{Z}$ and $\varphi(b'_{1})=(n_{1}/k)(g,0)$. Let $T$ be the automorphism of $\mathbb{Z}/n_{1}\mathbb{Z}\oplus\mathbb{Z}/n_{1}\mathbb{Z}$ defined on generators by $T(g,0)=(1,0)$ and $T(0,1)=(0,1)$. Now consider the isomorphism $f=T\circ\varphi$. Then $A'_{1}$ is isomorphic to $f(A'_{1})$ which is a subgroup contained in the subgroup $\left<(1,0)\right>$. Let  $b_1\in A_{1}$ so that $f(b_{1})=(1,k)$. Using these coordinates, we have $\left<f(b_{1})\right>\cap f(A'_{1})=\left<(n_{1}/k,0)\right>$.
Therefore, $\left<b_{1}\right>\cap A'_{1}=f^{-1}\left<(n_{1}/k,0)\right>=\left<b'_{1}\right>$.

The group $A'_2/B'_2$ is cyclic. In fact, let $\psi:A'_{2}\to A'/B'$ be the canonical projection defined by $x\mapsto x+B'$. Since $A'=A'_{1}+A'_{2}$ and $B'=B'_{1}+B'_{2}$ (both internal direct sums), we conclude that $\text{Ker}(\psi)=B'_{2}$. Thus $A'_2/B'_2$ is isomorphic to $\psi(A'_{2})$. By assumption $A'/B'$ is cyclic, therefore $\psi(A'_{2})$ also is. Since $A'_2/B'_2$ is cyclic, $B'_{2}$ is cyclic and $A'_{2}$ is the internal direct sum of the noncyclic $p$-subgroups of $A'$, then no cyclic subgroup of $A'_2$ properly contains $B'_2$. By Proposition \ref{prop:a2}, every element of $A_2$ is a multiple of a basis element; then we can take $b_2$ to be a basis element of $A_2$ so that some multiple of $b_2$ is $b'_2$.  Since $\left<b_{2}\right>\cap A'_{2}$ is a cyclic subgroup of $A'_{2}$ that contains $B'_{2}=\left<b'_{2}\right>$, we have that $\left<b_{2}\right>\cap A'_{2}=\left<b'_{2}\right>$. 

In the previous setting, let $B:=\left<b_{1}+b_{2}\right>$. Since $o(b_{1})$ and $o(b_{2})$ are coprime then $B=\left<b_{1}\right>+\left<b_{2}\right>$. Thus, 
$A'\cap B=(A'_{1}+A'_{2})\cap(\left<b_{1}\right>+\left<b_{2}\right>)=\left<b'_{1}\right>+\left<b'_{2}\right>=B'$. Now, note that

$$\Frac{A}{B}=\Frac{A_{1}+A_{2}}{\left<b_{1}\right>+\left<b_{2}\right>}\cong\Frac{A_{1}}{\left<b_{1}\right>}\oplus\Frac{A_{2}}{\left<b_{2}\right>}\cong\Frac{\mathbb{Z}/n_{1}\mathbb{Z}\oplus\mathbb{Z}/n_{1}\mathbb{Z}}{\left<(1,n_{1}/m)\right>}\oplus\Frac{A_{2}}{\left<b_{2}\right>}$$
Since $\gcd(n_{1},n_{2})=1$, and $(1,n_{1}/m)$ and $b_{2}$ are basis elements of $\mathbb{Z}/n_{1}\mathbb{Z}\oplus\mathbb{Z}/n_{1}\mathbb{Z}$ and  $A_{2}$ respectively, we finally conclude that $A/B$ is cyclic.
\end{proof}

\begin{prop}\label{prop:a01} Let $A$ be a finite Abelian group generated by two elements. Let $A'$ be a subgroup of $A$, and let $B'$ be a cyclic subgroup of $A'$ so that $A'/B'$ is also cyclic.  Then there exists a cyclic subgroup $B$ of $A$ such that $A/B$ is cyclic and $A'\cap B = B'$.\end{prop}
\begin{proof} Since $A$ is a finite Abelian group generated by two elements, there are positive integers $a$ and $n$ such that $a|n$ and $A\cong\mathbb{Z}/a\mathbb{Z}\oplus\mathbb{Z}/n\mathbb{Z}$. Assuming that  $A=\mathbb{Z}/a\mathbb{Z}\oplus\mathbb{Z}/n\mathbb{Z}$, there is a canonical monomorphism $i_{c}:A\to\mathbb{Z}/n\mathbb{Z}\oplus\mathbb{Z}/n\mathbb{Z}$,  defined on generators by 
$i_{c}(1,0)=(n/a,0)$ and $i_{c}(0,1)=(0,1).$ Thus, the group $A$ is isomorphic to the subgroup $\langle n/a\rangle\oplus\mathbb{Z}/n\mathbb{Z}$, where $\langle n/a\rangle$ is the cyclic subgroup of $\mathbb{Z}/n\mathbb{Z}$ generated by $n/a$.
Without loss of generality we may assume that $A$ is a subgroup of $\mathbb{Z}/n\mathbb{Z}\oplus\mathbb{Z}/n\mathbb{Z}$. Let $A'$ be a subgroup of $A$, and let $B'$ be a cyclic subgroup of $A'$ so that $A'/B'$ is also cyclic. Clearly $B'\subset A'\subset A\subset \mathbb{Z}/n\mathbb{Z}\oplus\mathbb{Z}/n\mathbb{Z}$. By Proposition \ref{prop:a3}, there exists $B$ a subgroup of  $\mathbb{Z}/n\mathbb{Z}\oplus\mathbb{Z}/n\mathbb{Z}$ such that $B$ and $(\mathbb{Z}/n\mathbb{Z}\oplus\mathbb{Z}/n\mathbb{Z})/B$ are cyclic, and $A'\cap B=B'$. Now, set $\tilde{B}=A\cap B$. Obviously $\tilde{B}$ is a cyclic subgroup of $A$ and $A'\cap \tilde{B}=A'\cap(A\cap B)=(A'\cap A)\cap B=A'\cap B=B'$.
Finally, using the canonical projection $\phi:A\to(\mathbb{Z}/n\mathbb{Z}\oplus\mathbb{Z}/n\mathbb{Z})/B$ defined by $x\mapsto x+B$, one sees that $\text{Ker}(\phi)=A\cap B=\tilde{B}$ and so $A/\tilde{B}$ is isomorphic to $\phi(A)$. Since $(\mathbb{Z}/n\mathbb{Z}\oplus\mathbb{Z}/n\mathbb{Z})/B$ is cyclic, $\phi(A)$ also is and therefore $A/\tilde{B}$ is cyclic.\end{proof}\

\begin{center}
List of Symbols.
\end{center}
$\mathbb{N}=\{0,1,2,\cdots\}$.\\
$\mathbb{Z}$ is the set of integer numbers.\\
$\Z^{+}$ is the set of positive integers.\\
$\mathbb{R}$ is set of real numbers.\\
$\mathbb{C}$ is the complex plane.\\
$\P^{1}$ is the Riemann Sphere.\\
$\Z^{2}$ is the 2-dimensional integer lattice.\\ 
$\Omega_{f}$ is critical set of the map $f$.\\
$V_{f}$ is critical value set of the map $f$.\\
$P_{f}$ is postcritical set of the Thurston map $f$.\\
$\Sigma_{f}$ is the pullback map induced by the Thurston map $f$.\\
$\sharp D$ is the number of elements in the set $D$.\\
$A=A_{1}+A_{2}$ internal direct sum of subgroups of $A$.\\
$|A|$ is the order of the group $A$.\\
$o(g)$ is the order of the element $g$.\\


\begin{thebibliography}{99}
\bibitem{BEKP}
X. Buff, A. Epstein, S. Koch, and K. Pilgrim,\begin{em} On Thurston's pullback map\end{em}. In $Complex$ $dynamics$, pp 561--583. AK Peters, Wellesley (2009). MR 2508269 (2010g:37071)

\bibitem{BM}M. Bonk and D. Meyer, \begin{em}Expanding Thurston maps\end{em}, Preprint (2010). ArXiv:1009.3647v1.

\bibitem{CFPP} J.W. Cannon, W.J. Floyd, W.R. Parry, and K.M. Pilgrim,\begin{em} Nearly Euclidean Thurston maps and finite subdivision rules\end{em}. Conform. Geom. Dyn. \textbf{16} (2012), 209--255(electronic).

\bibitem{DH} A. Douady and J.H. Hubbard,\begin{em} A proof of Thurston's topological characterization of rational functions\end{em}, Acta Math. \textbf{171} (1993), 263--297.

\bibitem{K} 
S. Koch, \begin{em}Teichm\"uller Theory and critically finite endomorphisms\end{em}. Advances in Mathematics, Vol. 248, 2013.

\bibitem{S}
E. A. Saenz Maldonado,\begin{em}On Nearly Euclidean Thurston maps,\end{em} Ph.D. Thesis, Virginia Tech, 2012.
\end{thebibliography}
\end{document}